\documentclass[11pt]{amsart}
\usepackage{amssymb, latexsym}

\numberwithin{equation}{section}
\newtheorem{Thm}[equation]{Theorem}
\newtheorem{Prop}[equation]{Proposition}
\newtheorem{Cor}[equation]{Corollary}
\newtheorem{Lem}[equation]{Lemma}

\theoremstyle{definition}

\newtheorem{Rmk}[equation]{Remark}

\begin{document}

\title [VOA associated to type $G$ affine Lie algebras]
{Vertex Operator Algebras \\ associated to type $G$ affine Lie algebras}
\author[J. D. Axtell]{Jonathan D. Axtell}
\address{Department of Mathematics,
University of Connecticut, Storrs, CT 06269-3009, U.S.A.}
\email{axtell@math.uconn.edu}
\author[K.-H. Lee]{Kyu-Hwan Lee}
\email{khlee@math.uconn.edu}

\begin{abstract}
In this paper, we study representations of the vertex operator algebra $L(k,0)$ at one-third admissible levels $k=- \frac 5 3 , -\frac 4 3, -\frac 2 3$ for the affine algebra of type $G_2^{(1)}$. We first determine singular vectors and then obtain a description of the associative algebra $A(L(k,0))$ using the singular vectors. We then prove that there are only finitely many irreducible $A(L(k,0))$-modules from the category $\mathcal O$. Applying the $A(V)$-theory, we prove that there are only finitely many irreducible weak $L(k,0)$-modules from the category $\mathcal O$ and that such an $L(k,0)$-module is completely reducible. Our result supports the conjecture made by Adamovi{\' c} and Milas in \cite{AM}.
\end{abstract}

\maketitle

\section*{Introduction}

Vertex operator algebras (VOA) are mathematical counterparts of conformal field theory.  An important family of examples comes from representations of affine Lie algebras. More precisely, if we let $\hat{\frak g}$ be an affine Lie algebra, the irreducible $\hat{\frak g}$-module $L(k,0)$ with highest weight $k \Lambda_0$,  $k \in \mathbb{C}$, is a VOA, whenever $k \ne - h^\vee$, the negative of the dual Coxeter number.

The representation theory of $L(k,0)$ varies depending on values of $k \in \mathbb C$. If $k$ is a positive integer, the VOA $L(k,0)$ has only finitely many irreducible modules which coincide with the irreducible integrable $\hat{ \frak g}$-modules of level $k$, and the category of $\mathbb Z_+$-graded weak $L(k,0)$-modules is semisimple. If $k \notin \mathbb Q$ or $k < - h^\vee$, categories of $L(k,0)$-modules are quite different from those corresponding to positive integer values. (For example, see \cite{KL1, KL2}.)

For some rational values of $k$, the category of weak $L(k,0)$-modules which are in the category $\mathcal O$ as $\hat{\frak g}$-modules has a similar structure as the category of $\mathbb Z_+$-graded weak modules for positive integer values. Such rational values are called {\em admissible levels}. This notion was defined in the important works of Kac and Wakimoto (\cite{KW1, KW2}). Various cases have been studied with different generality by many authors. Adamovi{\' c} studied the case of admissible half-integer levels for type $C_l^{(1)}$ \cite{A1}. The case of all admissible levels of type $A_1^{(1)}$ was studied by Adamovi{\' c} and Milas \cite{AM}, and by Dong, Li and Mason \cite{DLM}. In his recent papers  \cite{P, P1}, Per{\v s}e studied admissible half-integer levels for type $A_l^{(1)}$ and $B_l^{(1)}$. 

In these developments, the $A(V)$-theory has played an important role. The associative algebra $A(V)$ associated to a vertex operator algebra $V$ was introduced by I. Frenkel and Y. Zhu (see \cite{FZ, Z}). It was shown that the irreducible modules of $A(V)$ are in one-to-one correspondence with irreducible $\mathbb{Z}_+$-graded weak modules of $V$. This fact gives an elegant method for the classification of representations of $V$, and was exploited in the works mentioned above.

In this paper, we study one-third admissible levels $- \frac 5 3 \Lambda_0, -\frac 4 3 \Lambda_0, -\frac 2 3 \Lambda_0$ for type $G_2^{(1)}$ adopting the method of \cite{A1,AM,MP,P,P1}. We first determine singular vectors (Proposition \ref{prop-sing}) and then obtain a description of the associative algebra $A(L(k,0))$  in Theorem \ref{thm-zhu-image} using the singular vectors for $k= -\frac 5 3, - \frac 4 3, - \frac 2 3$. By constructing some polynomials in the symmetric algebra of the Cartan subalgebra, we find all the possible highest-weights for irreducible $A(L(k,0))$-modules from the category $\mathcal O$ (Proposition \ref{prop-finite}). As a result, in each case of $k= -\frac 5 3, - \frac 4 3, - \frac 2 3$, we prove that there are only finitely many irreducible $A(L(k,0))$-modules from the category $\mathcal O$. Then it follows from the one-to-one correspondence in $A(V)$-theory that there are only finitely many irreducible weak $L(k,0)$-modules from the category $\mathcal O$ (Theorem \ref{thm-main}). In the case of irreducible $L(k,0)$-modules, our result provides a complete classification (Theorem \ref{thm-class}).
We also prove that such an $L(k,0)$-module is completely reducible (Theorem \ref{thm-main-second}). Thus the VOA $L(k,0)$ is {\em rational in the category $\mathcal O$} for $k= -\frac 5 3, - \frac 4 3, - \frac 2 3$. This result supports the conjecture made by Adamovi{\' c} and Milas in \cite{AM}, which suggests that $L(k,0)$'s are rational in the category $\mathcal O$ for all admissible levels $k$.

Although some of our results may be generalized to higher levels $k$, the first difficulty is in the drastic growth of complexity in computing singular vectors, as one can see in Appendix A. It seems to be necessary to find a different approach to the problem for higher levels.  The first-named author will consider singular vectors for other admissible weights in his subsequent paper.

\subsection*{Acknowledgments} We thank A. Feingold and M. Primc for helpful comments.

\vskip 1cm

\section{Preliminaries}
\subsection{Vertex operator algebras}
Let $(V, Y, \mathbf{1}, \omega)$ be a vertex operator algebra (VOA).  This means that $V$ is a $\mathbb{Z}$-graded vector space, $V= \bigoplus_{n\in \mathbb{Z}} V_{n}$, $Y$ is the {\it vertex operator map}, $Y(\cdot, x): V  \rightarrow (\mbox{End } V ) [[x,x^{-1}]]$, $\mathbf{1} \in V_{0}$ is the {\em vacuum vector}, and $\omega \in V_{2}$ is the {\em conformal vector}, all of which satisfy the usual axioms.  See \cite{ DLM, FLM, LL} for more details.  
 By an ideal in the vertex operator algebra $V$ we mean a subspace $I$ of $V$ satisfying $Y(a,x)I \subseteq I[[x,x^{-1}]]$ for any $a \in V$.  Given an ideal $I$ in $V$ such that $\mathbf{1} \notin I$, $\omega \notin I$, the quotient $V/I$ naturally becomes a vertex operator algebra.  Let $(M, Y_M)$ be a weak module for the vertex operator algebra $V$.  We thus have a vector space $M$ and a map $Y_M(\cdot, x): V \rightarrow (\mbox{End }M)[[x,x^{-1}]]$, which satisfy the usual set of axioms (cf. \cite{DLM}).  For a fixed element $a \in V$, we write $Y_M(a,x) = \sum_{m \in\mathbb{Z}} a(m) x^{-m-1} $, and for the conformal element $\omega$ we write $Y_M(\omega,x) = \sum_{m \in\mathbb{Z}} \omega(m) x^{-m-1}  =  \sum_{m \in\mathbb{Z}} L_m x^{-m-2}$.  In particular, $V$ is a weak module over itself with $Y=Y_V$.

  A {\it $\mathbb Z_+$-graded weak $V$-module} is a weak $V$-module $M$ together with a $\mathbb Z_+$-gradation $M= \bigoplus_{n=0}^\infty M_n$ such that
\[ a(m) M_n \subseteq M_{n+r-m-1} \qquad \text{ for } a \in V_{r} \mbox{ and }m,n,r \in \mathbb Z, \] where $M_n =0$ for $n<0$ by definition. A weak $V$-module $M$ is called a {\it $V$-module} if $L_0$ acts semisimply on $M$ with a decomposition into $L_0$-eigenspaces $M= \bigoplus_{\alpha \in \mathbb C} M_{\alpha}$ such that for any $\alpha \in \mathbb C$, $\dim M_{\alpha} < \infty$ and $M_{\alpha +n}=0$ for $n \in \mathbb Z$ sufficiently small. 

We define bilinear maps $*: V \times V \rightarrow V$ and $\circ: V \times V \rightarrow V$ as follows.  For any homogeneous $a \in V_n$, we write $\mathrm {deg}( a)=n$, and  for any $b \in V$,  we define \[ a*b = \mathrm{Res}_x \frac {(1+x)^{\mathrm{deg} \, a}} x \ Y(a,x) b, \] and \[ a \circ b = \mathrm{Res}_x \frac {(1+x)^{\mathrm{deg} \, a}} {x^2} \ Y(a,x) b, \] and extend both definitions by linearity to $V \times V$.  Denote by $O(V)$ the linear span of elements of the form $a \circ b$, and by $A(V)$ the quotient space $V/O(V)$.  For $a \in V$, denote by $[a]$ the image of $a$ under the projection of $V$ onto $A(V)$. The map $a \mapsto [a]$ will be called {\em Zhu's map}. The multiplication $*$ induces the multiplication on $A(V)$, and $A(V)$ has a structure of an associative algebra.  This fact can be found in \cite{FZ,Z}. 

\begin{Prop} \cite{FZ} \label{prop-FZ}
Let $I$ be an ideal of the vertex operator algebra $V$ such that $\mathbf{1} \notin I$, $\omega \notin I$.  Then the associative algebra $A(V/I)$ is isomorphic to $A(V)/A(I)$, where $A(I)$ is the image of $I$ in $A(V)$.
\end{Prop}

Given a weak module $M$ and homogeneous $a \in V$, we recall that we write $Y_M(a,x) = \sum_{m\in \mathbb{Z}} a(m) x^{-m-1}$.  We define $o(a)= a({\mathrm{deg} \, a -1}) \in \mathrm{End} (M)$ and extend this map linearly to $V$.

\begin{Thm}\label{thm-Z} \cite{Z}
\begin{enumerate}
\item Let $M= \bigoplus_{n=0}^\infty M_n$ be a $\mathbb Z_+$-graded weak $V$-module.  Then $M_0$ is an $A(V)$-module defined as follows: \[ [a] \cdot v= o(a) v \] for any $a \in V$ and $v \in M_0$.

\item Let $U$ be an $A(V)$-module.  Then there exists a $\mathbb Z_+$-graded weak $V$-module $M$ such that the $A(V)$-modules $M_0$ and $U$ are isomorphic.

\item The equivalence classes of the irreducible $A(V)$-modules and the equivalence classes of the irreducible $\mathbb Z_+$-graded weak $V$-modules are in bijective correspondence.
\end{enumerate}
\end{Thm}

\medskip

\subsection{Affine Lie algebras}

Let $\frak g$ be a finite-dimensional simple Lie algebra over $\mathbb C$, with a triangular decomposition $\frak g = \frak n_- \oplus \frak h \oplus \frak n_+$.  Let $\Delta$ be the root system of $(\frak g, \frak h)$, $\Delta_+ \subset \Delta$ the set of positive roots, $\theta$ the highest root and $(\cdot , \cdot ): \frak g \times \frak g \rightarrow \mathbb C$ the Killing form, normalized by the condition $(\theta, \theta) =2$.  Denote by $\Pi = \{ \alpha_1, ... , \alpha_l \}$ the set of simple roots of $\frak g$, and by $\Pi^\vee = \{ h_1, ... , h_l \}$ the set of simple coroots of $\frak g$.  The affine Lie algebra $\hat{\frak g}$ associated to $\frak g$ is the vector space
\[ \hat{\frak g} = \frak g \otimes \mathbb C[t, t^{-1}] \oplus \mathbb C K \] equipped with the bracket operation\[ [a\otimes t^m, b\otimes t^n] = [a,b]\otimes t^{m+n} + m(a,b)\delta_{m+n,0}K, \hspace{1cm} a,b \in \mathfrak{g}, m,n \in \mathbb{Z}, \]  together with the condition that $K$ is a nonzero central element.

Let $h^\vee$ be the dual Coxeter number of $\hat{\frak g}$.  Let $\hat{\frak g} = \hat{\frak n}_- \oplus \hat{\frak h} \oplus \hat{\frak n}_+$ be the corresponding triangular decomposition of $\hat{\frak g}$.  Denote by $\widehat{\Delta}$ the set of roots of $\hat{\frak g}$, by $\widehat{\Delta}_+$ the set of positive roots of $\hat{\frak g}$, and by $\widehat{\Pi}$ the set of simple roots of $\hat{\frak g}$.  We also denote by $\widehat{\Delta}^{\mathrm {re}}$ the set of real roots of $\hat{\frak g}$ and let $\widehat{\Delta}^{\mathrm{re}}_+ = \widehat{\Delta}^{\mathrm{re}} \cap \widehat{\Delta}_+$.  The coroot corresponding to a real root $\alpha \in \widehat{\Delta}^{\mathrm{re}}$ will be denoted by $\alpha^\vee$.   Let $\widehat{Q}= \bigoplus_{\alpha \in \widehat{\Pi}} \mathbb Z \, \alpha$ be the root lattice, and let $\widehat{Q}_+= \bigoplus_{\alpha \in \widehat{\Pi}} \mathbb Z_+ \, \alpha \subset \widehat{Q}$. For any $\lambda \in \hat{\frak h}^*$, we set
\[ D(\lambda) = \left \{ \lambda - \alpha \ | \ \alpha \in \widehat{Q}_+ \right \}.  \]

We say that a $\hat{\frak g}$-module $M$ belongs to the {\it category $\mathcal O$} if the Cartan subalgebra $\hat{\frak h}$ acts semisimply on $M$ with finite-dimensional weight spaces and there exits a finite number of elements $\nu_1, ... , \nu_k \in \hat{\frak h}^*$ such that $\nu \in \bigcup_{i=1}^k D(\nu_i)$ for every weight $\nu$ of $M$. We denote by $M(\lambda)$ the Verma module for $\hat{\frak g}$ with highest weight $\lambda \in \hat{\frak h}^*$, and by $L(\lambda)$ the irreducible $\hat{\frak g}$-module with highest weight $\lambda$. 
Let $U$ be a $\frak g$-module, and let $k \in \mathbb C$. We set $\hat{\frak g}_+ = \frak g \otimes t \mathbb C[t]$ 
and $\hat{\frak g}_- = \frak g \otimes t^{-1} \mathbb C[t^{-1}]$. Let $\hat{\frak g}_+$
act trivially on $U$ and $K$ as scalar multiplication by $k$.  Considering $U$ as a $\frak g \oplus \mathbb C K \oplus \hat{\frak g}_+$-module, we have the induced $\hat{\frak g}$-module \[ N(k, U) = \mathcal{U}(\hat{\frak g}) \otimes_{\mathcal{U}(\frak g \oplus \mathbb C K \oplus \hat{\frak g}_+)} U. \]

For a fixed $\mu \in \frak h^*$, denote by $V(\mu)$ the irreducible highest weight $\frak g$-module with highest weight $\mu$. Denote by $P_+$ the set of dominant integral weights of $\frak g$, and by $\omega_1, ... , \omega_l \in P_+$ the fundamental weights of $\frak g$. We will write $N(k, \mu) = N(k, V(\mu))$. Denote by $J(k, \mu)$ the maximal proper submodule of $N(k, \mu)$ and $L(k, \mu)= N(k, \mu)/ J(k, \mu)$. We define $\Lambda_0 \in \hat{\frak h}^*$ by $\Lambda_0(K)=1$ and $\Lambda_0 (h)=0$ for any $h \in \frak h$. Then $N(k, \mu)$ is a highest-weight module with highest weight $k \Lambda_0+ \mu$, and a quotient of the Verma module $M(k \Lambda_0+\mu)$. We also obtain $L(k, \mu) \cong L(k \Lambda_0+ \mu)$.

\medskip

\subsection{Admissible weights}

 Let $\widehat{\Delta}^{\vee, \mathrm{re}}$ (respectively, $\widehat{\Delta}^{\vee, \mathrm{re}}_+$) be the set of real (respectively, positive real) coroots of $\hat{\frak g}$, and $\widehat{\Pi}^\vee$ the set of simple coroots.  For $\lambda \in \hat{\frak h}^*$,  we define
\[ \widehat{\Delta}^{\vee, \mathrm{re}}_\lambda = \{ \alpha^\vee \in \widehat{\Delta}^{\vee, \mathrm{re}} \, | \, \langle \lambda, \alpha^\vee \rangle \in \mathbb Z \}, \qquad \text{ and } \qquad 
  \widehat{\Delta}^{\vee, \mathrm{re}}_{\lambda , +} = \widehat{\Delta}^{\vee, \mathrm{re}}_\lambda \cap \widehat{\Delta}^{\vee, \mathrm{re}}_+,\] and we set \[
  \widehat{\Pi}^\vee_\lambda = \{ \alpha^\vee \in \widehat{\Delta}^{\vee, \mathrm{re}}_{\lambda , +}  \, | \, \alpha^\vee \text{ is not decomposable into a sum of elements from }  \widehat{\Delta}^{\vee, \mathrm{re}}_{\lambda , +} \}. \]

Let $\hat W$ denote the Weyl group of $\hat{\frak g}$.  For each $\alpha \in \widehat{\Delta}^{\mathrm{re}}$, we have a reflection $r_\alpha \in \hat W$. 
Define $\rho \in \hat{\frak h}^*$ in the usual way, and we recall the shifted action of an element $w \in \hat W$  on $\hat{\frak h}^*$, given by $w \cdot \lambda = w (\lambda + \rho) - \rho$.   
   
   A weight $\lambda \in \hat{\frak h}^*$ is called {\em admissible} if \[ \langle \lambda + \rho , \alpha^\vee \rangle \notin - \mathbb Z_{+} \quad \text{ for all } \alpha^\vee \in \widehat{\Delta}^{\vee, \mathrm{re}}_+ \qquad \text{ and } \qquad \mathbb Q \widehat{\Delta}^{\vee, \mathrm{re}}_\lambda = \mathbb Q \widehat{\Pi}^\vee. \]
The irreducible $\hat{\frak g}$-module $L(\lambda)$ is called {\em admissible} if the weight $\lambda \in \hat{\frak h}^*$ is admissible. 
Given a $\hat{\frak g}$-module $M$ from the category $\mathcal O$, we call a weight vector $v \in M$ a {\it singular vector} if $\hat{\frak n}_+ . v = 0$.

\begin{Prop} \label{prop-KW} \cite{KW1}
Let $\lambda$ be an admissible weight. Then
\[ L(\lambda) = M(\lambda) \big / \left ( \sum_{\alpha^\vee \in \widehat{\Pi}^\vee_\lambda} \mathcal{U}(\hat{\frak g}) v_{\alpha} \right ) , \] where $v_{\alpha} \in M(\lambda)$ is a singular vector of weight $r_{\alpha} \cdot \lambda$. 
\end{Prop}

\begin{Prop} \label{prop-KW2}\cite{KW2} Let $M$ be a $\hat{\frak g}$-module from the category $\mathcal O$. If every irreducible subquotient $L(\nu)$ of $M$ is admissible, then $M$ is completely reducible.
\end{Prop}

\medskip

\subsection{$N(k,0)$ and $L(k,0)$ as VOA's}
We identify the one-dimensional trivial $\frak g$-module $V(0)$ with $\mathbb C$. Write $\mathbf{1}=1 \otimes 1 \in N(k,0)$. The $\hat{\frak g}$-module $N(k,0)$ is spanned by the elements of the form \[ a_1(-n_1-1) \cdots a_m(-n_m-1) \mathbf{1},\] where $a_1, ... , a_m \in \frak g$ and $n_1, ... , n_m \in \mathbb Z_{+}$, with $a(n)$ denoting the element $a \otimes t^n$ for $a \in \frak g$ and $n \in \mathbb Z$.  

The vector space $N(k,0)$ admits a VOA structure, which we now describe.
The vertex operator map $Y(\cdot , x) : N(k,0) \rightarrow {\rm{End}} (N(k,0)) [[x, x^{-1}]]$ is uniquely determined by defining $Y(\mathbf{1}, x)$ to be the identity operator on $N(k,0)$ and \[ Y( a(-1) \mathbf{1}, x) = \sum_{n \in \mathbb Z} a(n) x^{-n-1} \quad \text{ for } a \in \frak g . \]   In the case that $k \neq - h^\vee$, the module $N(k,0)$ has a conformal vector \[ \omega = \frac 1 {2(k+h^\vee)} \sum_{i=1}^{\dim \frak g} (a^i (-1))^2 \mathbf{1},\] where $\{a^i\}_{i=1, ... , \dim \frak g}$ is an arbitrary orthonormal basis of $\frak g$ with respect to the normalized Killing form $(\cdot , \cdot )$.  Then it is well known that the quadruple $( N(k,0), Y, \mathbf{1}, \omega)$ defined above is a vertex operator algebra.

\begin{Prop} \cite{FZ}   The associative algebra $A(N(k,0))$ is canonically isomorphic to $\mathcal{U}(\frak g)$. The isomorphism is given by $F: A(N(k,0)) \rightarrow \mathcal{U}(\frak g)$, \[ F( [a_1(-n_1-1) \cdots a_m(-n_m-1) \mathbf{1}]) = (-1)^{n_1+ \cdots +n_m} a_1 \cdots a_m,\] for  $a_1, ... , a_m \in \frak g$ and  $n_1, ... , n_m \in \mathbb Z_{+}$.
\end{Prop}

Since every $\hat{\frak g}$-submodule of $N(k,0)$ is also an ideal in the VOA $N(k,0)$, the module $L(k,0)$ is a VOA for any $k \neq -h^\vee.$

\begin{Prop} \label{prop-important} \cite{P}
Assume that the maximal $\hat{\frak g}$-submodule of $N(k,0)$ is generated by a singular vector $v_0$. Then we have
\[ A(L(k,0)) \cong \mathcal{U}(\frak g) \big / \langle F([v_0]) \rangle , \] where $\langle F([v_0]) \rangle$ is the two-sided ideal of $\mathcal{U}(\frak g)$ generated by $F([v_0])$. In particular, a $\frak g$-module $U$ is an $A(L(k,0))$-module if and only if $F([v_0]) U=0$.
\end{Prop}

\vskip 1 cm

\section{Affine Lie algebra of type $G_2^{(1)}$}

\subsection{Admissible weights}
Let   \[   \Delta =  \left \{ \begin{array}{lll} \pm \frac{1}{\sqrt{3}}(\epsilon_1-\epsilon_2),& \pm \frac{1}{\sqrt{3}}(\epsilon_1-\epsilon_3),& \pm \frac{1}{\sqrt{3}}(\epsilon_2-\epsilon_3), \\ \pm \frac{1}{\sqrt{3}}(2\epsilon_1-\epsilon_2-\epsilon_3),&  \pm \frac{1}{\sqrt{3}}(2\epsilon_2-\epsilon_1-\epsilon_3),& \pm \frac{1}{\sqrt{3}}(2\epsilon_3-\epsilon_1-\epsilon_2) \end{array} \right \}    \] be the root system of type $G_2$.  We fix the set of positive roots \[    \Delta_+ =  \left \{ \begin{array}{lll} \frac{1}{\sqrt{3} }(\epsilon_1-\epsilon_2) ,& \frac{1}{\sqrt{3} }(\epsilon_3-\epsilon_1), & \frac{1}{\sqrt{3} }(\epsilon_3-\epsilon_2), \\ \frac{1}{\sqrt{3} }(-2\epsilon_1+\epsilon_2+\epsilon_3) , & \frac{1}{\sqrt{3} }(-2\epsilon_2+\epsilon_1+\epsilon_3), & \frac{1}{\sqrt{3} }(2\epsilon_3 - \epsilon_1-\epsilon_2) \end{array} \right \}.   \]  Then the simple roots are $\alpha= \frac 1 {\sqrt{3}} ( \epsilon_1-\epsilon_2)$ and $\beta= \frac 1 {\sqrt{3}} (-2\epsilon_1+\epsilon_2+\epsilon_3)$, and the highest root is $\theta = \frac 1 {\sqrt{3}} (2\epsilon_3-\epsilon_1-\epsilon_2) = 3\alpha + 2\beta$.
Let $\mathfrak{g}$ be the simple Lie algebra over $\mathbb{C}$, associated with the root system of type $G_2$.  Let $E_{10},E_{01},F_{10},F_{01},H_{10},H_{01}$ be Chevalley generators of $\mathfrak{g}$, where $E_{10}$ is a root vector for $\alpha$, $E_{01}$ is a root vector for $\beta$, and so on.  We fix the root vectors:

\begin{equation} \label{xx}
\begin{split}
E_{11}& = [E_{10}, E_{01}],\\
E_{21}& = \frac{1}{2} [E_{11},E_{10}] = \frac{1}{2} [[E_{10}, E_{01}], E_{10}],\\
E_{31}& = \frac{1}{3} [ E_{21}, E_{10}] = \frac{1}{6} [[[E_{10}, E_{01}], E_{10}], E_{10}]\\
E_{32}& = [ E_{31}, E_{01}]=\frac{1}{6} [[[[E_{10}, E_{01}], E_{10}], E_{10}],E_{01}],\\
F_{11}& = [F_{01}, F_{10}],\\
F_{21}& = \frac{1}{2} [F_{10} , F_{11}]=\frac{1}{2} [F_{10} , [F_{01}, F_{10}]],\\
F_{31}& =\frac{1}{3} [F_{10} ,F_{21}]=\frac{1}{6} [F_{10} ,[F_{10} , [F_{01}, F_{10}]]],\\
F_{32}& =  [F_{01}, F_{31}]=\frac{1}{6} [F_{01}, [F_{10} ,[F_{10} , [F_{01}, F_{10}]]]].
\end{split}
\end{equation}\\
We set $H_{ij}  = [E_{ij}, F_{ij}]$  for any positive root $i \alpha+ j \beta \in \Delta_+$. Then one can check that $H_{ij} $ is the coroot corresponding to $i \alpha+ j \beta$, i.e. $H_{ij}= (i \alpha+ j \beta)^\vee$. For a complete multiplication table, we refer the reader to Table 22.1 in \cite[p.346]{FH}, where we have
\[ \begin{array}{llllll}
X_1= E_{10}, & X_2= E_{01}, & X_3= E_{11}, & X_4= - E_{21}, & X_5= - E_{31}, & X_6= - E_{32}, \\
Y_1= F_{10},  & Y_2= F_{01},  & Y_3= F_{11},  & Y_4= - F_{21},  & Y_5=- F_{31},  & Y_6= - F_{32}.
\end{array} \]

All admissible weights for arbitrary affine Lie algebras have been completely classified in \cite{KW2}.  The next proposition provides a description of the ``vacuum" admissible weights for $G_2^{(1)}$ at one-third levels.  This is a special case of Proposition 1.2 in \cite{KW3}.  We provide a proof for completeness.

\begin{Lem} \label{lem-adm}
The weight $\lambda_{3n+i} = (n-2+\frac{i}{3})\Lambda_0$ is admissible for $n \in \mathbb{Z}_{+}, i = 1,2,$ and we have
\[ \widehat{ \Pi }_{\lambda_{3n+i}}^{\vee} = \{ (\delta- (2 \alpha+\beta) )^\vee, \alpha^\vee, \beta^\vee \},   \] where $\delta$ is the canonical imaginary root.
Furthermore,
\begin{eqnarray}
   & & \langle \lambda_{3n+i}+\rho, \gamma^\vee \rangle = 1 \quad \textit{ for } \gamma=\alpha,\beta ; \nonumber  \\
   & & \langle \lambda_{3n+i}+\rho, (\delta-(2 \alpha+\beta))^\vee \rangle = 3n+i+1 \quad  \textit{ for } i=1,2.  \nonumber
\end{eqnarray}
\end{Lem}

\begin{proof}
We have to show
\begin{align}
&\langle  \lambda_{3n+i} + \rho, {\gamma}^\vee \rangle \notin -\mathbb{Z}_+  \qquad \text{    for any }\gamma \in \widehat{\Delta}_+^{\mathrm{re}} \nonumber \\
\text{ and } \qquad \qquad  & \mathbb{Q} \widehat{\Delta}_{\lambda_{3n+i}}^{\vee, \mathrm{re}}  = \mathbb{Q} \widehat{ \Pi }^{\vee}.     \nonumber
\end{align}

Any positive real root ${\gamma} \in \widehat{ \Delta}_+^{\mathrm{re}} $ of $\hat{\mathfrak{g} }$ is of the form ${\gamma} = \bar{\gamma} + m \delta$, for $m > 0$ and $\bar{\gamma} \in \Delta$, or $m=0$ and $\bar{\gamma} \in \Delta_+$.  Denote by $\bar{\rho}$ the sum of fundamental weights of $\mathfrak{g}$.  Then we can choose $\rho = h^\vee \Lambda_0 + \bar{\rho} = 4 \Lambda_0 + \bar{\rho}$.

We have
\begin{eqnarray*}
 \langle \lambda_{3n+i} + \rho, {\gamma}^\vee \rangle   & = & \bigl{\langle} \bigl{(}n+2+\tfrac{i}{3}\bigr{)} \Lambda_0 + \bar{\rho}, (\bar{\gamma}+m\delta)^\vee \bigr{\rangle} \\
 & = & \tfrac{2}{(\bar{\gamma},\bar{\gamma})} \bigl{(}m\bigl{(}n+2+\tfrac{i}{3} \bigr{)} + (\bar{\rho}, \bar{\gamma}) \bigr{)}.
\end{eqnarray*}
If $m=0$, then it is trivial that $\langle \lambda_{3n+i}, {\gamma}^\vee \rangle \notin - \mathbb{Z}_+$.
Suppose that $m \ge 1$.  If $(\bar{\gamma}, \bar{\gamma}) = 2$ and $m$ $\not \equiv 0\ (\mathrm{mod}\ 3)$, then $\langle \lambda_{3n+i} + \rho, {\gamma}^\vee \rangle \notin - \mathbb{Z}_+$.
If $(\bar{\gamma}, \bar{\gamma}) = 2$, and $m$ $ \equiv 0\ (\mathrm{mod}\ 3)$, then $m\ge3$, and since $(\bar{\rho}, \bar{\gamma}) \ge -3$ for any $\bar{\gamma} \in \Delta$, we have
\[  \langle \lambda_{3n+i} + \rho, {\gamma} ^\vee \rangle = m \bigl{(} n+2+\tfrac{i}{3} \bigr{)} + (\bar{\rho}, \bar{\gamma}) \ge 3 \bigl{(}n +2+ \tfrac{1}{3}) -3= 3n+4 \ge 4,\]
which implies $\langle \lambda_{3n+i} + \rho, {\gamma}^\vee \rangle  \notin - \mathbb{Z}_+$.
If $(\bar{\gamma}, \bar{\gamma}) = \frac{2}{3}$, then $(\bar{\rho}, \bar{\gamma}) \ge -\frac{5}{3}$.  We have
\[   \langle \lambda_{3n+i} + \rho, {\gamma} ^\vee \rangle = 3 \bigl{(} m \bigl{(} n+2+\tfrac{i}{3} \bigr{)} + (\bar{\rho}, \bar{\gamma})\bigr{)} \ge  3 \bigl{(}  n+\tfrac{7}{3}  + (\bar{\rho}, \bar{\gamma} )\bigr{)}  \ge 3 \bigl{(}n + \tfrac{7}{3} -\tfrac{5}{3} \bigr{)}= 3n+2 \ge 2,\]
which implies $\langle \lambda_{3n+i} + \rho, {\gamma}^\vee \rangle  \notin - \mathbb{Z}_+$.
Thus, $\langle \lambda_{3n+i} + \rho, {\gamma}^\vee \rangle \notin -\mathbb{Z}_+$ for any ${\gamma} \in \widehat{\Delta}_+^{\mathrm{re}}.$  

One can easily see that \begin{eqnarray*} \widehat{\Delta}^{\vee, \mathrm{re}}_{\lambda_{3n+i}, +} &=& 
\{ m\delta + \bar{\gamma} | m>0, m \equiv 0 \ (\mathrm{mod}\ 3), (\bar{\gamma}, \bar{\gamma})= 2 \} \\ &  & \cup \
\{ m\delta + \bar{\gamma} | m>0, (\bar{\gamma}, \bar{\gamma})=  2/ 3 \}  \cup \Delta_+ , \end{eqnarray*}
Then we obtain \[ \widehat{ \Pi }_{\lambda_{3n+i}}^{\vee} = \{ (\delta- (2 \alpha+ \beta) )^\vee, \alpha^\vee, \beta^\vee \},   \] 
and we see that $\mathbb{Q} \widehat{\Delta}_{\lambda_{3n+i}}^{\vee, \mathrm{re}}  =  \mathbb{Q} \widehat{ \Pi }_{\lambda_{3n+i}}^{\vee}  =  \mathbb{Q} \widehat{ \Pi }^{\vee}. $
Through direct calculations, we get
\[ \begin{aligned}
    & \langle \lambda_{3n+i}+\rho,\gamma^\vee \rangle = 1 \text{ for } \gamma=\alpha,\beta, \text { and } \\
    & \langle \lambda_{3n+i}+\rho, (\delta-(2 \alpha+\beta))^\vee \rangle = 3n+i+1.
\end{aligned} \]

\end{proof}

\vskip 1 cm

\subsection{Singular Vectors}

In what follows, let $\hat{\mathfrak{g}}$ be the affine Lie algebra of type $G_2^{(1)}$ and  $\mathcal{U}(\hat{ \frak g})$ its universal enveloping algebra.  

 We write $X^i(-m)=X(-m)^i$ for elements in $\mathcal{U}(\hat{ \frak g})$. We set \[ \begin{array}{rl}
  a = & E_{21}(-1),\\
  
b = & E_{31}(-1) E_{11}(-1)
     -\ E_{32}(-1) E_{10}(-1), \\
     
c = & E^2_{31}(-1) E_{01}(-1)
    -\ E_{32}(-1) E_{31}(-1) H_{01}(-1)
    -\ E^2_{32}(-1) F_{01}(-1), \\
  w = & E_{31}(-1) E_{32}(-2)   
    -\ E_{32}(-1) E_{31}(-2),     \end{array}
 \] and define \[ u =  \tfrac 1 3  a^2 - b, \quad \text{ and } \quad v =  \tfrac 2 9 a^3 -ab -3c .\]

The following proposition determines singular vectors for the first three admissible weights, i.e. $- \frac 5 3 \Lambda_0, -\frac 4 3 \Lambda_0, -\frac 2 3 \Lambda_0$, respectively. 

\begin{Prop} \label{prop-sing}
  The  vector $v_k \in N(k,0)$ is a singular vector for the given value of $k$:
\[ v_k = \left \{ \begin{array}{ll}
u  . \mathbf{1} & \mbox{ for } \ k = -\frac{5}{3}, \\
(v+w) . \mathbf{1} & \mbox{ for } \ k = -\frac{4}{3}, \\
u(v - w) .  \mathbf{1} & \mbox{ for } \ k = -\frac{2}{3} .
\end{array} \right .  \]
\end{Prop}

The proof will be given in the Appendix A. As one can see in the proof, the computational difficulty increases as the level $k$ goes up. A different approach will be used in a subsequent work of the first-named author on higher levels.

\subsection{Descripton of Zhu's algebra}

\begin{Prop}\label{prop-sing-max}
The maximal $\hat{\mathfrak{g}}$-submodule $J(k,0)$ of $N(k,0)$   is generated by the vector $v_k$ for $k=-\frac{5}{3}$, $-\frac{4}{3}$, $-\frac{2}{3}$, respectively, where $v_k$'s are given in Proposition \ref{prop-sing}.
\end{Prop}

\begin{proof} Let $\lambda_{3n+i}=(-2+n+\tfrac{i}{3})\Lambda_0= k \Lambda_0$ as before. It follows from Proposition \ref{prop-KW} and Lemma \ref{lem-adm} that the maximal submodule of the Verma module $M(\lambda_{3n+i})$ is generated by three singular vectors with weights \[r_{\delta-(2\alpha+\beta)} \cdot \lambda_{3n+i}, \quad r_\alpha \cdot \lambda_{3n+i}, \quad r_\beta \cdot \lambda_{3n+i} , \qquad \text{respectively}.  \] We consider the three cases \[ n=0,i=1, k=-5/3;  \qquad n=0,i=2, k=-4/3; \qquad n=1,i=1, k=-2/3 . \]
In each case, there is a singular vector $u_{k} \in M(\lambda_{3n+i})$ of weight $r_{\delta-(2\alpha+\beta)}.\lambda_{3n+i},$
whose image  under the projection of $M(\lambda_{3n+i})$ onto $N(k, 0)$ is the singular vector $v_k$ given in Proposition \ref{prop-sing}.

The other singular vectors have weights \[ \begin{array}{l} r_{\alpha}\cdot \lambda_{3n+i}  =  \lambda_{3n+i} - \langle \lambda_{3n+i}+\rho,\alpha^{\vee}\rangle\alpha = \lambda_{3n+i} - \alpha , \quad \mbox{ and }\\ r_{\beta}\cdot \lambda_{3n+i} = \lambda_{3n+i} - \langle \lambda_{3n+i}+\rho,\beta^{\vee}\rangle\beta = \lambda_{3n+i} - \beta,\end{array} \] so the images of these vectors under the projection of $M(\lambda_{3n+i})$ onto $N(k, 0)$ are $0$ from the definition.  Therefore the maximal submodule of $N(k, 0)$ is generated by the singular vector $v_k$, i.e. $J(k, 0) = \mathcal{U}(\hat{\mathfrak{g}})v_k$.
\end{proof}

Now we consider the image of a singular vector $v_k$ under Zhu's map \[ [\cdot ]: N(k,0)\rightarrow A(N(k,0)) \cong \mathcal{U}(\mathfrak{g}),\]  which is defined in Section 1.  We recall that the vertex algebra $N(k,0)$ is (linearly) isomorphic to the associative algebra $\mathcal{U}(\hat{\mathfrak{g}}_-)$.  We thus have an induced map from $\mathcal{U}(\hat{\mathfrak{g}}_-)$ to $\mathcal{U}(\mathfrak{g})$ and  a commutative diagram of linear maps:
\medskip
\begin{center}
$\begin{array}{ccc}
\mathcal{U}(\hat{\mathfrak{g}}_-) & \simeq &  N(k,0)  \\
\downarrow & &\downarrow \\
\mathcal{U}(\mathfrak{g}) & \simeq & A(N(k,0))
\end{array}$ 
\end{center}
\medskip
We will identify $N(k,0)$ with $\mathcal{U}(\mathfrak{\hat{g}}_-)$ and $A(N(k,0))$ with $\mathcal{U}(\mathfrak{g})$.  We have:
$$
\begin{array}{l}
[a] = E_{21},\\

[b] = E_{31} E_{11} - E_{32} E_{10} ,\\

[c] =E^2_{31}E_{01} - E_{32}E_{31}H_{01} - E^2_{32}F_{01}.
\end{array}$$
We also have:
\begin{equation} \label{eqn-bracket} \begin{array}{l}
[u]= \tfrac{1}{3} [a]^2 - [b] , \\

[v]= \tfrac{2}{9} [a]^3 - [a][b] - 3 [c] , \\

[w]=0 ,  \\

 [u(v-w)] = [u][v]= \tfrac{2}{27}[a]^5-\tfrac{5}{9}[a]^3[b]-[a]^2[c]+[a][b]^2+3[b][c] .
\end{array} \end{equation}

The following theorem is now a consequence of Propositions \ref{prop-important} and \ref{prop-sing-max}.

\begin{Thm} \label{thm-zhu-image}
The associative algebra $A(L(k,0))$ is isomorphic to $\mathcal{U}(\mathfrak{g})/I_{k}$,  where $ I_{k}$ is the two-sided ideal of $\mathcal{U}(\mathfrak{g})$ generated by the vector $[v_k]$, where
\[ [v_k] = \left \{ \begin{array}{ll}
[u] & \quad \mbox{ for } k = -\frac{5}{3}, \\

[v] & \quad \mbox{ for } k = -\frac{4}{3}, \\

[uv] & \quad \mbox{ for } k = -\frac{2}{3}.
\end{array} \right .  \]
\end{Thm}

\vskip 1 cm 
 
 \section{Irreducible modules} 
 In this section we adopt the method from \cite{A1,AM,MP,P,P1} in oder to classify irreducible $A(L(k, 0))$-modules from the category $\mathcal{O}$ by solving certain systems of polynomial equations.
 
 \subsection{Modules for associative algebra $A(L(k, 0))$. }
 
 Denote by $_L$ the adjoint action of $\mathcal{U}(\mathfrak{g})$ on $\mathcal{U}(\mathfrak{g})$ defined by $X_Lf = { [ X, f]}$ for $ X \in \mathfrak{g}$ and $f \in \mathcal{U}(\mathfrak{g})$.  We also write $(ad \, X)f = X_Lf= [X,f]$. Then $ad \, X$ is a derivation on $\mathcal{U}(\mathfrak{g})$. Let $R(k)$ be a $\mathcal{U}(\mathfrak{g})$-submodule of $\mathcal{U}(\mathfrak{g})$ generated by the vector $[v_{k}]$, where $[v_k]$ is given in Theorem \ref{thm-zhu-image}.  It is straightforward to see that $R(k)$ is an irreducible finite-dimensional $\mathcal{U}(\mathfrak{g})$-module isomorphic to $V((3k+7)(2\alpha+\beta))$.  Let $R(k)_0$ be the zero-weight subspace of $R(k)$.
 
 \begin{Prop} \cite{A1,AM}
 Let $V(\mu)$ be an irreducible highest weight $\mathcal{U}(\mathfrak{g})$-module with highest weight vector $v_\mu$ for $\mu \in \mathfrak{h}^*$.  Then the following statements are equivalent:
\begin{enumerate}
\item  $V(\mu)$ is an $A(L(k, 0))$-module,
\item   $R(k) \cdot V(\mu) = 0$,
\item   $R(k)_0 \cdot v_\mu = 0$. 
\end{enumerate}
 \end{Prop}
  
 Let $r\in R(k)_0$.  Then there exists a unique polynomial $p_r \in S(\mathfrak{h})$, where $S(\mathfrak{h})$ is the symmetric algebra of $\mathfrak h$,  such that \[r \cdot v_\mu = p_r(\mu) v_\mu. \]  Set $\mathcal{P}(k)_0 = \{p_r \, | \, r \in \mathcal {R}(k)_0 \}$.  Then we have:
 
 \begin{Cor}  \label{cor-biject} There is a bijective correspondence between \begin{enumerate} \item the set of irreducible $A(L(k, 0) )$-modules $V(\mu)$ from the category $\mathcal{O}$, and \item the set of weights $\mu \in \mathfrak{h}^*$ such that $p(\mu)=0$ for all $p \in \mathcal{P}(k)_0$. \end{enumerate}
 \end{Cor}

\medskip

\subsection{Polynomials in  $\mathcal{P}(k)_0$}

We now determine some polynomials in the set $\mathcal{P}(k)_0$ for the cases $k= - \frac 5 3$, $k= - \frac 4 3$, $k= - \frac 2 3$, respectively. We will use some computational lemmas which we collect and prove in Appendix B.

\begin{Lem}[Case: $k=- \frac 5 3$] \label{lem-first-case} We let 
\[ (1) \ q(H) =  H_{21}( H_{21} +2 ),  \quad (2) \ p_1(H) = H_{10}(H_{10}-1), \quad \text{and } \quad (3)\ p_2(H)=  \tfrac{1}{3}H_{11}(H_{11}-1) +3H_{01} .\]
 Then $q(H), p_1(H), p_2(H) \in \mathcal{P}(-\frac 5 3)_0$.
 \end{Lem}

\begin{proof}
(1) We show that $(E_{21}^2 F_{21}^4)_L [u] \equiv C\, q(H)  \ (\mathrm{mod} \ \mathcal{U}(\mathfrak{g})\mathfrak{n}_+) $ for some $C \ne 0$.  Using Lemma \ref{lem-secondstep} and Lemma \ref{lem-thirdstep}, we have 
\begin{eqnarray*}
(E_{21}^2F_{21}^4)_L [u] & = &(E_{21}^2F_{21}^4)_L (\tfrac{1}{3}[a]^2-[b])\\
 &\equiv & 4!2!(\tfrac{1}{3} H_{21}(H_{21}-1) + H_{21})   \equiv 4!2!\tfrac{1}{3} H_{21}(H_{21}+2)  \quad (\mathrm{mod} \ \mathcal{U}(\mathfrak{g})\mathfrak{n}_+)  ,
\end{eqnarray*}
which is what we wanted to show.

(2) We will show that $(E_{10}^2 F_{31}^2)_L [u] \equiv C\,  p_1(H)  \ (\mathrm{mod} \ \mathcal{U}(\mathfrak{g})\mathfrak{n}_+) $ for some $C \ne 0$.  We again use Lemma \ref{lem-secondstep} and Lemma \ref{lem-thirdstep} to obtain:
\[
(E_{10}^2F_{31}^2)_L (\tfrac{1}{3}[a]^3-[b]) \equiv (2!)^2\ \tfrac{1}{3} H_{10}(H_{10}-1) \equiv \tfrac{4}{3}p_1(H)  \quad (\mathrm{mod} \ \mathcal{U}(\mathfrak{g})\mathfrak{n}_+) .
\]

(3) In this case we show that $(E_{11}^2 F_{32}^2)_L [u] \equiv C\, p_2(H)  \ (\mathrm{mod} \ \mathcal{U}(\mathfrak{g})\mathfrak{n}_+) $ for some $C \ne 0$.  Similarly to the first two cases we compute:
\begin{eqnarray*}
(E_{11}^2F_{32}^2)_L(\tfrac{1}{3}[a]^2-[b]) &\equiv & (2!)^2\ ( \tfrac{1}{3} H_{11}(H_{11}-1) + 3 H_{01}) \\ & \equiv & C\, p_2(H)  \quad (\mathrm{mod} \ \mathcal{U}(\mathfrak{g})\mathfrak{n}_+) .
\end{eqnarray*}
 \end{proof}
 
We now give polynomials for the next  case.

\begin{Lem}[Case: $k=- \frac 4 3$] \label{lem-second-case} Let 
\begin{enumerate}
\item $q(H) =  \frac{2}{9}H_{21}(H_{21}-1)(H_{21}-2) +\ H_{21} (H_{21}-2) +3 H_{01} (H_{01}+2)$,
\item $p_1(H) =  H_{10}(H_{10}-1)(H_{10}-2) $,
\item $p_2(H)  =  \frac{2}{9}H_{11}(H_{11}-1)(H_{11}-2) +6 H_{01} H_{32}$.
\end{enumerate}
 Then $p_1(H),p_2(H), q(H) \in \mathcal{P}(-\frac 4 3)_0$.
 \end{Lem}

 \begin{proof}
(1) We show that $(E^3_{21}F^6_{21})_L [v] \equiv C q(H) \ (\mathrm{mod} \ \mathcal{U}(\mathfrak{g})\mathfrak{n}_+) $ for some constant $C \neq 0$.
By Lemma \ref{lem-thirdstep}, we have:
\begin{eqnarray*}
(E_{21}^3F_{21}^6)_L [v]& =& (E_{21}^3F_{21}^6)_L(\tfrac{2}{9}[a]^3 -[a][b]-3[c])\\
& \equiv & -3!6! \tfrac{2}{9} H_{21}(H_{21}-1)(H_{21}-2) -\tfrac{3!}{2!}\tfrac{6!}{4!} (H_{21}-2)(E_{21}^2F_{21}^4)_L [b] \\ & & -3 (E_{21}^3F_{21}^6)_L[c]  \quad (\mathrm{mod} \ \mathcal{U}(\mathfrak{g})\mathfrak{n}_+) .
\end{eqnarray*}
By Lemma \ref{lem-secondstep}, we thus have:
\begin{eqnarray*}
(E_{21}^3F_{21}^6)_L [v]& \equiv & -3!6! \tfrac{2}{9}H_{21}(H_{21}-1)(H_{21}-2) + 3!6! (H_{21}-2) H_{21}+ 3!6! H_{01}(H_{01}+2) \\
& \equiv & C q(H)  \quad (\mathrm{mod} \ \mathcal{U}(\mathfrak{g})\mathfrak{n}_+) .
\end{eqnarray*}

(2) We will show that $(E_{10}^3F_{31}^3)_L [v] \equiv C p_1(H) \ (\mathrm{mod} \ \mathcal{U}(\mathfrak{g})\mathfrak{n}_+) $ for some constant $C \neq 0$.
Using Lemma \ref{lem-thirdstep}, we obtain:
\begin{eqnarray*}
& & (E_{10}^3F_{31}^3)_L (\tfrac{2}{9}[a]^3 - [a][b] - 3 [c]) \\ & \equiv & \tfrac{2}{9} (3!)^2 H_{10}(H_{10}-1)(H_{10}-2) + \tfrac{3!}{2!}\tfrac{3!}{2!}(H_{10}-2) (E_{10}^2F_{31}^2)_L [b]  -3 (E_{10}^3F_{31}^3)_L [c] \quad (\mathrm{mod} \ \mathcal{U}(\mathfrak{g})\mathfrak{n}_+) .
\end{eqnarray*}
By Lemma \ref{lem-secondstep}, we thus have
\begin{eqnarray*}
(E_{10}^3F_{31}^3)_L (\tfrac{2}{9}[a]^3 - [a][b] - 3 [c])& \equiv & \tfrac{2}{9} (3!)^2 H_{10}(H_{10}-1)(H_{10}-2)  \\
&\equiv & C p_1(H)  \quad (\mathrm{mod} \ \mathcal{U}(\mathfrak{g})\mathfrak{n}_+) .
\end{eqnarray*}

(3) Finally, we show that $(E_{11}^3F_{32}^3)_L [v]  \equiv C p_2(H) \ (\mathrm{mod} \ \mathcal{U}(\mathfrak{g})\mathfrak{n}_+) $ for some constant $C \neq 0$.  Since $H_{11}+H_{31} = 2 H_{32}$, we have 
\begin{eqnarray*}
(E_{11}^3F_{32}^3)_Lv '& \equiv& \tfrac{2}{9} (3!)^2\ H_{11}(H_{11}-1)(H_{11}-2) - \tfrac{3!}{2!}\tfrac{3!}{2!} (H_{11}-2) (E_{11}^2F_{32}^2)_L [b] - 3 (E_{11}^3F_{32}^3)_L [c]\\
&\equiv & (3!)^2\ (\tfrac{2}{9}H_{11}(H_{11}-1)(H_{11}-2) + 3 (H_{11}-2) H_{01} + 3 H_{01} (H_{31}+2))\\
&\equiv & (3!)^2\ (\tfrac{2}{9}H_{11}(H_{11}-1)(H_{11}-2) +  6 H_{01}H_{32} ) \\
& \equiv & C p_2(H)  \quad (\mathrm{mod} \ \mathcal{U}(\mathfrak{g})\mathfrak{n}_+) .
\end{eqnarray*} 
  \end{proof} 

The last case is presented below.

 \begin{Lem}[Case: $k=-\frac 2 3$]\label{lem-third-case} We let 
 \begin{eqnarray*}
 q(H) &= & \tfrac{2}{27}H_{21}(H_{21}-1)(H_{21}-2)(H_{21}-3)(H_{21}-4) +\tfrac{5}{9} H_{21} (H_{21}-2)(H_{21}-3)(H_{21}-4)\\
      & &+\ (H_{21}-3)(H_{21}-4) H_{01}( H_{01}+2) + 2 H_{21}(H_{21}-4)(H_{11}-1)  \\
      & & +\ 2 (H_{21}-4)H_{10}(H_{10}-1)-\ 6(H_{21}-4)H_{01}(H_{01}+1) \ +\ 6 (H_{21}-3)H_{01} (H_{01}+2), \\
 p_1(H) &=& H_{10}(H_{10}-1)(H_{10}-2)(H_{10}-3)(H_{10}-4),\\
   p_2(H) &=&  \tfrac{2}{27}H_{11}(H_{11}-1)(H_{11}-2)(H_{11}-3)(H_{11}-4) +\tfrac{5}{3}(H_{11}-2)(H_{11}-3)(H_{11}-4)H_{01}\\	  		& &		+\ (H_{11}-3)(H_{11}-4)H_{01}(H_{31}+2) 	+18 (H_{11}-4)H_{01}(H_{01}-1)	\\ & & -2(H_{11}-3)(H_{11}-4)H_{01}	+18 H_{01}(H_{01}-1)(H_{31}+2) .
\end{eqnarray*}
  Then $p_1(H),p_2(H), q(H) \in \mathcal{P}(-\frac 2 3)_0$.
 \end{Lem}
 
\begin{proof}
First recall from (\ref{eqn-bracket}) that \[ [u(v-w)] = [u][v]= \tfrac{2}{27}[a]^5-\tfrac{5}{9}[a]^3[b]-[a]^2[c]+[a][b]^2+3[b][c] .\]  We will show that $(E_{21}^5F_{21}^{10})_L([u][v]) \equiv - 5!10! q(H)  \ (\mathrm{mod} \ \mathcal{U}(\mathfrak{g})\mathfrak{n}_+) $.  

Using Lemmas \ref{lem-product}, \ref{lem-firststep}, we have:
\begin{align*}
(F_{21}^{10})_L([u][v]) =& (F_{21}^{10})_L (\tfrac{2}{27}[a]^5 -\tfrac{5}{9}[a]^3[b]-[a]^2[c]+[a][b]^2+3[b][c])\\
                        =&\tfrac{2}{27} \tfrac{10!}{(2!)^5} (-2)^5 F_{21}^5 - \tfrac{5}{9} \tfrac{10!}{(2!)^3 4!} (-2)^3 F_{21}^3\ (F_{21}^{4})_L[b] - \tfrac{10!}{(2!)^2 6!} (-2)^2 F_{21}^2\  (F_{21}^{6})_L[c]\\
                        & + \tfrac{10!}{2!8!} (-2) F_{21}\ (F_{21}^{8})_L[b]^2 + 3\ (F_{21}^{10})[b][c]\\ 
												= & -\tfrac{2}{27} 10!  F_{21}^5 + \tfrac{5}{9} \tfrac{10!}{ 4!}  F_{21}^3\ (F_{21}^{4})_L[b]\\
                        & - \tfrac{10!}{ 6!} F_{21}^2\  (F_{21}^{6})_L[c] - \tfrac{10!}{8!}  F_{21}\ (F_{21}^{8})_L[b]^2 + 3\ (F_{21}^{10})[b][c] .
              \end{align*}
                        Now using Lemma \ref{lem-poly}, we obtain:
\begin{align*}
\tfrac{1}{10!}(E_{21}^5F_{21}^{10})_L([u][v]) = -&\tfrac{2}{27}\ 5!H_{21}(H_{21}-1)(H_{21}-2)(H_{21}-3)(H_{21}-4)\\
  &+\tfrac{5}{9}\ \tfrac{5!}{2!}\  (H_{21}-2)(H_{21}-3)(H_{21}-4)\ \tfrac{1}{4!}(E_{21}^2F_{21}^{4})_L[b]\\
  &-\tfrac{5!}{3!} (H_{21}-3)(H_{21}-4)\ \tfrac{1}{6!}(E_{21}^3F_{21}^{6})_L[c]\\
  & - \tfrac{5!}{4!}(H_{21}-4)\ \tfrac{1}{8!}(E_{21}^4F_{21}^{8})_L[b]^2 + 3\ \tfrac{1}{10!}(E_{21}^5F_{21}^{10})_L([b][c]).
\end{align*}
Combining this with Lemmas \ref{lem-secondstep}, \ref{lem-thirdstep}, \ref{lem-fourthstep}, we obtain:
\begin{align*}
\tfrac{1}{10!}(E_{21}^5F_{21}^{10})_L([u][v])\equiv &-\tfrac{2}{27}\ 5!H_{21}(H_{21}-1)(H_{21}-2)(H_{21}-3)(H_{21}-4)\\
  &+\tfrac{5}{9}\ \tfrac{5!}{2!}\  (H_{21}-2)(H_{21}-3)(H_{21}-4)\ (-2)H_{21}\\
  &-\tfrac{5!}{3!} (H_{21}-3)(H_{21}-4)\ 3! H_{01}(H_{01}+2)\\
  & - \tfrac{5!}{4!}(H_{21}-4)\ 4! (2 H_{21}H_{11} + 2 H_{10}(H_{10}-1) - 6 H_{01}(H_{01}+1))\\
  & + 3\ 5!(-2) H_{01}(H_{01}+2)(H_{21}-3) \\
  \equiv &-5! q(H)  \quad (\mathrm{mod} \ \mathcal{U}(\mathfrak{g})\mathfrak{n}_+) .  \end{align*}
The proofs for $p_1(H)$ and $p_2(H)$ are similar, and we omit the details.
 \end{proof}

\medskip
\subsection{Finiteness of the number of irreducible modules}

We are now able to obtain the following result for the associative algebra $A(L(k,0))$.  For convenience, if $\mu \in \frak h^*$, we write $\mu_{ij} = \mu(H_{ij})$.    We will identify $\mu \in \frak h^*$ with the pair $(\mu_{10}, \mu_{01})$.

 \begin{Prop}\label{prop-finite}
 There are finitely many irreducible $A(L(k,0))$-modules from the category $\mathcal{O}$ for each of $k=-\frac 5 3, -\frac 4 3, - \frac 2 3$. Moreover, the possible highest weights $\mu=(\mu_{10}, \mu_{01})$ for irreducible $A(L(k,0))$-modules are as follows:
 \begin{enumerate}
 \item if $k=-\frac{5}{3}$, then $\mu=(0,0),  (0, -\tfrac{2}{3})$ or $(1,-\tfrac{4}{3})$;
\item  if $k=-\frac{4}{3}$, then $\mu=(0,0),  (0, -\tfrac{2}{3}), (0,-\tfrac{1}{3}), (1,0), (1, -\tfrac{4}{3})$ or $(2,-\tfrac{5}{3})$; 
\item  if $k=-\frac{2}{3}$, then $\mu=(0,0), (0,-\tfrac{2}{3}), (0,-\tfrac{1}{3}), (0,\tfrac{1}{3})$, $(0, 1), (1,0),  (1,-\tfrac{4}{3}), (1,-\tfrac{2}{3})$, \\ \phantom{LLLLLlLLLLLLLLL}$(2,0), (2,-\tfrac{5}{3}), (2,-\tfrac{4}{3})$ or $(4,-\tfrac{7}{3})$.
\end{enumerate}
 \end{Prop}
   
 \begin{proof}
(1) It follows from Corollary \ref{cor-biject} that highest weights $\mu \in \mathfrak{h}^*$ of irreducible $A(L(-\frac{5}{3},0))$-modules satisfy $p(\mu)=0$ for all $p \in \mathcal{P}_0(-\frac 5 3)$.  Lemma \ref{lem-first-case} implies that $p_1(\mu)= p_2(\mu) = q(\mu) = 0$ for such weights $\mu$.  Let $\mu \in \mathfrak{h}^*$.  The equation $p_1(\mu)=0$ is  \[ \mu_{10} ( \mu_{10}-1) = 0,\] which implies $\mu_{10}= 0$ or 1.

First suppose $\mu_{10}=0$.  Then from $q(\mu)=0$ we must have $\mu_{01}= 0$ or $-\frac{2}{3}$.  Similarly, from $p_2(\mu)=0$, we also get $\mu_{01}= 0$ or $-\frac{2}{3}$.  So the weight $\mu$ must be of the form $\mu = (\mu_{10}, \mu_{01}) = (0, 0)$ or $(0,-\frac{2}{3})$ in this case.  Now suppose $\mu_{10}=1$.  The equation $q(\mu)=0$ gives $\mu_{01}=-\frac{2}{3}$ or $-\frac{4}{3}$, and the equation $p_2(\mu)=0$ gives $\mu_{01}= 0$ or $-\frac{4}{3}$.  So the only possibility is $\mu = (\mu_{10},\mu_{01}) =(1, -\frac{4}{3})$.
Altogether, this gives only three possible weights $\mu$ such that $p_1( \mu) = p_2( \mu) = q( \mu) = 0$: \[\mu = (\mu_{10}, \mu_{01}) = (0, 0), (0,-\tfrac{2}{3}), \mbox{ or } (1, -\tfrac{4}{3}).  \]  
 
 (2) Similarly to the part (1), we use the polynomials of Lemma \ref{lem-second-case}. Using a computer algebra system, we calculate the common zeros of the polynomials $q(H), p_1(H), p_2(H)$ to obtain the following list of possible highest weights: \[ \mu=(\mu_{10}, \mu_{01})= (0,0),  (0, -\tfrac{2}{3}), (0,-\tfrac{1}{3}), (1,0), (1, -\tfrac{4}{3}), \mbox{ or } (2,-\tfrac{5}{3}) .\]
 
 (3) For this part, we use Lemma \ref{lem-third-case}. Using a computer algebra system, we again compute the common zeros of the polynomials $q(H), p_1(H), p_2(H)$ to obtain the following list of possible highest weights: \[ \begin{array}{ll} \mu=(\mu_{10}, \mu_{01}) =   & (0,0), (0,-\tfrac{2}{3}), (0,-\tfrac{1}{3}), (0,\tfrac{1}{3}), (0, 1), \\ & (1,0),  (1,-\tfrac{4}{3}), (1,-\tfrac{2}{3}), (2,0), (2,-\tfrac{5}{3}), (2,-\tfrac{4}{3}), \mbox{ or } (4,-\tfrac{7}{3}). \end{array} \]
 \end{proof}  
 
Now we apply the $A(V)$-theory (Theorem \ref{thm-Z}), and obtain our main result in the following theorem. 
\begin{Thm} \label{thm-main} There are finitely many irreducible weak modules from the category $\mathcal{O}$ for each of the following simple vertex operator algebras: $L(-\frac{5}{3},0)$, $L(-\frac{4}{3},0)$, $L(-\frac{2}{3},0)$.
\end{Thm}

\begin{Rmk} \label{rmk-VOA}
 This theorem provides further evidence for the conjecture of Adamovi{\' c} and Milas in \cite{AM}, mentioned in the introduction.
Furthermore, if $L(\lambda)$ is an irreducible module of the VOA $L(k,0)$, for $k= -\frac{5}{3}, -\frac{4}{3}$, or $-\frac{2}{3}$, then we recall from Section 1.2 that we must have $L(\lambda) \cong L(k\Lambda_0, \mu)$ for the values of
 $\mu \in \frak h^*$ given in  Proposition \ref{prop-finite}.
\end{Rmk}

In the case of irreducible $L(k,0)$-modules, we obtain a complete classification. We state this result in the following proposition and theorem.

\begin{Prop}
The complete list of irreducible finite-dimensional $A(L(k,0))$-modules $V(\mu)$ for each $k$ is as follows:
\begin{enumerate}
\item if $k=- \frac 5 3$, then $V(\mu)= V(0)$,
\item if $k=- \frac 4 3$, then $V(\mu) = V(0)$ or $V(\omega_1)$,
\item if $k=-\frac 2 3$, then $V(\mu)=V(0), V(\omega_1), V(\omega_2)$, or $V(2 \omega_1)$,
\end{enumerate}
where $\omega_1, \omega_2$ are the fundamental weights of $\frak g$.
\end{Prop}

\begin{proof}
Among the list of weights in Proposition \ref{prop-finite}, we need only to consider dominant integral weights, i.e. those weights $\mu=(m_1, m_2)$ with $m_1, m_2 \in \mathbb Z_+$. Notice that the weights of the singular vectors $[v_k]$ are $2 \omega_1$, $3 \omega_1$ and $5 \omega_1$, respectively. Considering the set of weights of $V(\mu)$  listed above, we  see that each singular vector $[v_k]$ annihilates the corresponding modules $V(\mu)$. Now the proposition follows from Proposition \ref{prop-important}.
\end{proof}

We again apply the $A(V)$-theory (Theorem \ref{thm-Z}), and obtain the following theorem. 
\begin{Thm} \label{thm-class}
The complete list of irreducible $L(k,0)$-modules $L(k,\mu)$ for each $k$ is as follows:
\begin{enumerate}
\item if $k=- \frac 5 3$, then $L(k,\mu)= L(k,0)$,
\item if $k=- \frac 4 3$, then $L(k,\mu) = L(k,0)$ or $L(k,\omega_1)$,
\item if $k=-\frac 2 3$, then $L(k,\mu)=L(k,0), L(k,\omega_1), L(k,\omega_2)$, or $L(k,2 \omega_1)$.
\end{enumerate}
\end{Thm}

\medskip

\subsection{Semisimplicity of weak modules from the category $\mathcal{O}$}

In this subsection we show that the category of weak $L(k,0)$-modules from the category $\mathcal{O}$ is semisimple.

\begin{Lem} \label{lem-adm2}  Assume that $\lambda = k \Lambda_0 + \mu$ for $k= - \frac 5 3, - \frac 4 3, - \frac 2 3$, where  $\mu \in \frak h^*$ is one of the values given in  Proposition \ref{prop-finite} for each $k$. Then the weights $\lambda$ are admissible.
\end{Lem}

\begin{proof}  The proof is essentially the same as Lemma \ref{lem-adm}. Let us write $\widehat{\Pi}_0^\vee = \{ (\delta-(2\alpha + \beta))^{\vee}, \alpha^{\vee}, \beta^{\vee}\}$, $\widehat{\Pi}_1^\vee = \{ (\delta-(3\alpha + \beta))^{\vee}, \alpha^{\vee}, (\alpha+\beta)^{\vee}\}$, and $\widehat{\Pi}_2^\vee = \{ (\delta-\theta)^{\vee}, \alpha^{\vee}, (\alpha + \beta)^{\vee}\}$. 
Since the proof for the other cases are similar, we consider only the case $k=-\tfrac{5}{3}$.  From Lemma \ref{lem-adm}, we already know that $\lambda = -\tfrac{5}{3}\Lambda_0 + \mu$ is admissible for $\mu = (0,0)$, with $\widehat \Pi_{\lambda}^{\vee} = \widehat \Pi_0^\vee$.

If $\mu = (0, -\tfrac{2}{3})$, we have to show that 
\[ \langle -\tfrac{5}{3}\Lambda_0 + \mu + \rho, \gamma^{\vee}\rangle \notin -\mathbb{Z}_+ \mbox{   for any } \gamma \in \widehat{\Delta}_+^{\mathrm{re}} \quad \mbox{  and } \quad  \mathbb{Q} \widehat{\Delta}_{\lambda}^{\vee, \mathrm{re}} = \mathbb{Q} \widehat{\Pi}^{\vee}.
\]

Recall that $\rho = 4 \Lambda_0 + \bar{\rho}$; also $\gamma \in \widehat{\Delta}_+^{\mathrm{re}}$ must have the form $\gamma = \bar{\gamma} + m \delta$, for $m>0$ and $\bar{\gamma} \in \Delta$, or $m= 0$ and $\bar{\gamma} \in \Delta_+$.  We then have:
\begin{align*}
\langle -\tfrac{5}{3}\Lambda_0 + \mu + \rho, \gamma^{\vee}\rangle &= \langle(  \tfrac{7}{3} \Lambda_0 + \mu + \bar{\rho}, (\bar{\gamma}+m\delta)^{\vee} \rangle\\
&= \tfrac{2}{(\bar{\gamma}, \bar{\gamma})}\, \tfrac{7}{3}m + \langle \mu, \bar{\gamma}^{\vee} \rangle + \langle \bar{\rho}, \bar{\gamma}^{\vee} \rangle.
\end{align*}
We may then check that $\langle -\frac{5}{3}\Lambda_0 + \mu + \rho, \gamma^{\vee}\rangle \ge \frac{1}{3}$, so that $\langle -\frac{5}{3}\Lambda_0 + \mu + \rho, \gamma^{\vee}\rangle \notin -\mathbb{Z}_+$.  
One may also verify that $\widehat \Pi_{\lambda}^{\vee} = \widehat \Pi_1^\vee$ so that $\mathbb{Q} \widehat{\Delta}_{\lambda}^{\vee, \mathrm{re}} = \mathbb{Q} \widehat{\Pi}^{\vee}$.

Similarly, one can show that $\lambda = -\tfrac{5}{3}\Lambda_0 + \mu$ is admissible for $\mu= (1,-\frac{4}{3})$ and that $\widehat \Pi_{\lambda}^{\vee} = \widehat \Pi_2^\vee$.
\end{proof}

\begin{Thm}  \label{thm-main-second} Let $M$ be a weak $L(k,0)$-module from the category $\mathcal{O}$, for $k=-\frac{5}{3}, -\frac{4}{3}$, or $-\frac{2}{3}$.  Then $M$ is completely reducible.
\end{Thm}

\begin{proof} Let $L(\lambda)$ be an irreducible subquotient of $M$.  Then $L(\lambda)$ is an $L(k,0)$-module, and we see from Remark \ref{rmk-VOA} that $\lambda$ must be a weight of the form $k \Lambda_0 + \mu$, where $\mu$ is given in Proposition \ref{prop-finite} for $k = -\frac 5 3, - \frac 4 3, -\frac 2 3$, respectively.  From Lemma \ref{lem-adm2} it follows that such a $\lambda$ is admissible.  Now Proposition \ref{prop-KW2} implies that $M$ is completely reducible.

\end{proof}

\vskip 1cm

\appendix

\section{Proof of Proposition 2.3}
In this appendix, we prove Proposition \ref{prop-sing}. We first give a few lemmas.

\begin{Lem} \label{lem-commute} \hfill
\begin{enumerate} 
\item  We have
\[ \begin{array}{l}\  [a, E_{10}(0)]= 3 E_{31}(-1),  \qquad  [b, E_{10}(0)]= 2 E_{31}(-1) E_{21}(-1), \\ \   [c , E_{10}(0)] = E_{32}(-1)E_{31}(-1)E_{10}(-1) - E_{31}^2(-1) E_{11}(-1) ,\\ \
[u, E_{10}(0)] =0 , \ [v, E_{10}(0)] =0 , \ [w, E_{10}(0)] =0 .
\end{array} \]

\item Each of the elements $a,b,c,u, v, w \in \mathcal U(\hat{ \frak g})$ commutes with $E_{01}(0)$.

\end{enumerate}
\end{Lem} 

\begin{proof}
(1) Using the multiplication table in (\ref{xx}), it is easy to see  $[a, E_{10}(0)]= 3 E_{31}(-1)$. Next, we have
\begin{eqnarray*}
[b, E_{10}(0)] &=& [ E_{31}(-1) E_{11}(-1)
     - E_{32}(-1) E_{10}(-1) , E_{10}(0) ] \\ &=& E_{31}(-1)[E_{11}(-1), E_{10}(0)] + [E_{31}(-1), E_{10}(0) ] E_{11}(-1) \\ & & - E_{32}(-1) [ E_{10}(-1) , E_{10}(0) ]- [ E_{32}(-1)  , E_{10}(0) ] E_{10}(-1) \\ &=& 2 E_{31}(-1) E_{21}(-1) .
\end{eqnarray*}
Starting with the definition
\[ [c, E_{10}(0)] = [ E^2_{31}(-1) E_{01}(-1)
    - E_{32}(-1) E_{31}(-1) H_{01}(-1)
    - E^2_{32}(-1) F_{01}(-1),  E_{10}(0) ] ,\] we consider each term separately and obtain
 \begin{eqnarray*}
 & & [ E^2_{31}(-1) E_{01}(-1),  E_{10}(0) ] \\ &=& E_{31}^2(-1)[E_{01}(-1), E_{10}(0)] + E_{31}(-1)[E_{31}(-1), E_{10}(0) ] E_{01}(-1) +  [E_{31}(-1), E_{10}(0) ] E_{31}(-1) E_{01}(-1) \\  & =& -E_{31}^2(-1) E_{11}(-1),
\end{eqnarray*}
\begin{eqnarray*}
 & & [ E_{32}(-1)E_{31}(-1) H_{01}(-1),  E_{10}(0) ] \\ &=&  E_{32}(-1)E_{31}(-1)[H_{01}(-1), E_{10}(0)] + E_{32}(-1)[E_{31}(-1), E_{10}(0) ] H_{01}(-1) \\ & & \phantom{LLLLLLLLLLLLLLLLLLLLLLLLLLLLLL} +  [E_{32}(-1), E_{10}(0) ] E_{31}(-1) H_{01}(-1) \\  & =& - E_{32}(-1) E_{31}(-1) E_{10}(-1),
\end{eqnarray*}
and 
\begin{eqnarray*}
 & & [E^2_{32}(-1) F_{01}(-1),  E_{10}(0) ] \\ & =&  E_{32}^2(-1)[F_{01}(-1), E_{10}(0)] + E_{32}(-1)[E_{32}(-1), E_{10}(0) ] F_{01}(-1) +  [E_{32}(-1), E_{10}(0) ] E_{32}(-1) F_{01}(-1) \\  & =& 0 .
 \end{eqnarray*}
Therefore, we obtain \[ [c , E_{10}(0)] = E_{32}(-1)E_{31}(-1)E_{10}(-1) - E_{31}^2(-1) E_{11}(-1) .\]
Next, we get
\begin{eqnarray*}
[u, E_{10}(0)] &=& \tfrac 1 3 [a^2, E_{10}(0)] -[ b , E_{10}(0)] \\ &=& \tfrac 1 3 a [a, E_{10}(0)] + \tfrac 1 3 [a, E_{10}(0)]a  -[ b , E_{10}(0)] \\ & =& E_{21}(-1) E_{31}(-1)  + E_{31}(-1) E_{21}(-1)  -2 E_{31}(-1) E_{21}(-1) =0 ,
\end{eqnarray*}
and
\begin{eqnarray*}
[v, E_{10}(0)] &=& \tfrac 2 9 [a^3, E_{10}(0)] - [ab, E_{10}(0)] -3 [c,E_{10}(0)] \\ & =&  2 E_{21}^2(-1)E_{31}(-1) -a[b, E_{10}(0)]- [a, E_{10}(0)] b -3 [c,E_{10}(0)]\\ & =&  2 E_{21}^2(-1)E_{31}(-1) - 2 E_{21}(-1)E_{31}(-1)E_{21}(-1) \\ & & -3 E_{31}(-1) \left \{ E_{31}(-1)E_{11}(-1) - E_{32}(-1)E_{10}(-1) \right \} \\ & & - 3 \left \{ E_{32}(-1)E_{31}(-1)E_{10}(-1) - E_{31}^2(-1) E_{11}(-1) \right \} =0 .
\end{eqnarray*}
Finally, it is easy to see $[w, E_{10}(0)]=0$. 

(2) The equalities $[a, E_{01}(0)]=0$, $[b, E_{01}(0)]=0$,  $[c, E_{01}(0)]=0$ can be proved as in the part (1), and we omit the details. Then it immediately follows that $[u, E_{01}(0)]=0$ and $[v, E_{01}(0)]=0$. 
Since $w = \tfrac{1}{3}[a,b]$, we also obtain $[w, E_{01}(0)]=0$.
\end{proof}

\begin{Lem} \label{lem-commute-F}
We have
\begin{eqnarray*}  \ [a, F_{32}(1)] &=& -F_{11}(0), \\ \ [b,  F_{32}(1)] &=&  E_{31}(-1) F_{21}(0)- E_{11}(-1)F_{01}(0) - E_{10}(-1)H_{32}(0) +(K+1) E_{10}(-1), 
\\   \ [c ,  F_{32}(1)] &=& E_{32}(-1)E_{31}(-1)F_{32}(0)+ E_{32}(-1)H_{01}(-1)F_{01}(0)-2 E_{32}(-1)F_{01}(-1)H_{32}(0)   \\ & & + (2K+2) E_{32}(-1)F_{01}(-1)  + E_{31}^2(-1)F_{31}(0) - 2 E_{31}(-1)E_{01}(-1)F_{01}(0) \\ & & - E_{31}(-1)H_{01}(-1)H_{32}(0)+(K+1) E_{31}(-1)H_{01}(-1),
\\ \ [u, F_{32}(1)]& =& - \left (K+ \tfrac 5 3 \right ) E_{10}(-1) - E_{31}(-1)F_{21}(0)  - \tfrac 2 3 E_{21}(-1)F_{11}(0) \\ & &
+ E_{11}(-1)F_{01}(0) + E_{10}(-1)H_{32}(0) , 
\\ \  [v,  F_{32}(1)] &=& - E_{32}(-1)E_{10}(-1)F_{11}(0) - 3 E_{32}(-1)F_{01}(-1) \\ & & + \tfrac 4 3 E_{31}(-2) + E_{31}(-1)E_{11}(-1)F_{11}(0) -E_{31}(-1)H_{11}(-1) \\ & & - \tfrac 2 3 a^2 F_{11}(0) - \tfrac 1 3 a E_{10}(-1) -a [b, F_{32}(1)] -3 [ c, F_{32}(1)] , \\ \ [w,  F_{32}(1)] &=&  - E_{32}(-2)F_{01}(0) +E_{32}(-1)F_{01}(-1) \\ & & -  E_{31}(-2)H_{32}(0)+ E_{31}(-1)H_{32}(-1) +K E_{31}(-2)   .
\end{eqnarray*} 
\end{Lem} 

\begin{proof}
We only prove the equalities for $[b, F_{32}(1)]$ and $[u, F_{32}(1)]$. The other equalities can be proved similarly. We obtain
\begin{eqnarray*}
 [b, F_{32}(1)] &=& [  E_{31}(-1) E_{11}(-1)
     - E_{32}(-1) E_{10}(-1) ,  F_{32}(1)] \\ &=& E_{31}(-1)[E_{11}(-1), F_{32}(1) ] + [E_{31}(-1), F_{32}(1) ] E_{11}(-1) \\ & & - E_{32}(-1) [ E_{10}(-1) , F_{32}(1) ]- [ E_{32}(-1)  , F_{32}(1) ] E_{10}(-1) \\ &=& E_{31}(-1) F_{21}(0) - F_{01}(0)E_{11}(-1) - \left \{ H_{32}(0)-K \right \} E_{10}(-1) \\ &=&
      E_{31}(-1) F_{21}(0)- E_{11}(-1)F_{01}(0) - E_{10}(-1)H_{32}(0) +(K+1) E_{10}(-1),
\end{eqnarray*}
and
\begin{eqnarray*}
 [u, F_{32}(1)]  &=& \tfrac 1 3 a [a,  F_{32}(1)] + \tfrac 1 3 [a,  F_{32}(1)]a  -[ b ,  F_{32}(1)] \\ &=& - \tfrac 2 3 E_{21}(-1)F_{11}(0) - \tfrac 2 3  E_{10}(-1) 
 \\ & & - E_{31}(-1) F_{21}(0)+ E_{11}(-1)F_{01}(0) + E_{10}(-1)H_{32}(0) - (K+1) E_{10}(-1)
\\ & =&  - \left (K+ \tfrac 5 3 \right ) E_{10}(-1) - E_{31}(-1)F_{21}(0)  - \tfrac 2 3 E_{21}(-1)F_{11}(0) \\ & & + E_{11}(-1)F_{01}(0) + E_{10}(-1)H_{32}(0)  .
 \end{eqnarray*}

\end{proof}

We need one more lemma. 
\begin{Lem} \label{lem-zero} We have the following commutator relations:
\[  [H_{32}(0), v-w ] =3(v-w) , \quad [ F_{01}(0),v-w]=0,   \]
\begin{eqnarray*}
[ F_{11}(0),v-w] &=& \left ( \tfrac 1 3 a^2 -2 b \right ) E_{10}(-1) + a E_{31}(-1)H_{10}(-1) \\ & & -5a E_{31}(-2) +5 E_{31}(-1)E_{21}(-2) \\ & & +3 E_{31}^2(-1)F_{10}(-1)+3E_{32}(-1)E_{31}(-1)F_{11}(-1)-3aE_{32}(-1)F_{01}(-1), \\ \
 [ F_{21}(0),v-w] &=& \left (-\tfrac 2 3 a^2 + b \right ) H_{21}(-1) + \tfrac 2 3  a E_{21}(-2)-2 a E_{31}(-1)F_{10}(-1)-2a E_{32}(-1)F_{11}(-1) \\ & & +3 E_{31}(-1)E_{11}(-1)H_{01}(-1)+3E_{32}(-1)E_{10}(-1)H_{01}(-1) \\ & & -6 E_{31}(-1)E_{10}(-1)E_{01}(-1) +6 E_{32}(-1)E_{11}(-1)F_{01}(-1) \\ & &  + 4 E_{11}(-1)E_{31}(-2) - 4 E_{10}(-1)E_{32}(-2).\end{eqnarray*}
\end{Lem}

\begin{proof}
Since the proofs of the other equalities are similar, we only provide a proof for $F_{11}(0)$. 
We first have
\[
[F_{11}(0), v-w] = [F_{11}(0), \tfrac 2 9 a^3 - ab -3 c- w] .
\]
Considering each term separately, we get
\begin{eqnarray*}
\ [F_{11}(0) ,  a^3] & = & 6 a^2 E_{10}(-1) - 18 aE_{31}(-2) ,\\
\ [F_{11}(0), ab]&=& [F_{11}(0), a]b +a[F_{11}(0), b] \\ &=& -2 E_{10}(-1)b - a\left \{ -E_{31}(-1)H_{11}(-1)+a E_{10}(-1)-3E_{32}(-1)F_{01}(-1)  \right \}, \\ \ [F_{11}(0), c] &=& - E_{31}^2(-1)F_{10}(-1) + aE_{31}(-1) H_{01}(-1) \\ & & -E_{32}(-1)E_{31}(-1)F_{11}(-1) + 2 a E_{32}(-1)F_{01}(-1) , \\ \   [F_{11}(0), w] &=& a E_{31}(-2) -  E_{31}(-1)E_{21}(-2).
\end{eqnarray*}
Using two more relations \[ [E_{10}(-1), b] = -2 E_{31}(-1)E_{21}(-2) \quad \text{ and } \quad H_{11}= H_{10}+3H_{01},\] one can now obtain the result for $[F_{11}(0), v-w]$.

\end{proof}

We now prove Proposition \ref{prop-sing}. For convenience, we state the proposition again:

\begin{Prop}
  The vector $v_k \in N(k,0)$ is a singular vector for the given value of $k$:
\[ v_k = \left \{ \begin{array}{ll}
u  . \mathbf{1} & \mbox{ for } \ k = -\frac{5}{3}, \\
(v+w) . \mathbf{1} & \mbox{ for } \ k = -\frac{4}{3}, \\
u(v - w) .  \mathbf{1} & \mbox{ for } \ k = -\frac{2}{3} .
\end{array} \right .  \]
\end{Prop}

\begin{proof}

To show that each vector $v_k$ is a singular vector, it suffices to check that $E_{10}(0).v_k = 0$, $E_{01}(0).v_k=0$, and $F_{32}(1).v_k=0$ for each $k$. Assume that $k = - \frac 5 3$. By Lemma  \ref{lem-commute}, we obtain \[ E_{10}(0).v_k = E_{10}(0) u  . \mathbf{1} = - [u, E_{10}(0)] . \mathbf 1 =0 ,\]  and similarly we get $E_{01}(0).v_k =0$. Now we consider $F_{32}(1)$ and obtain by Lemma \ref{lem-commute-F}
\[ F_{32}(1).v_k = - [u, F_{32}(1)].\mathbf 1 =0 .\]

Assume that $k = -\frac 4 3$. It follows from Lemma \ref{lem-commute} that $E_{10}(0).v_k=0$ and $E_{01}(0).v_k=0$. We also obtain  from Lemma \ref{lem-commute-F}
\begin{eqnarray*}
F_{32}(1).v_k &=& -[v+w, F_{32}(1)]  \\ &=& 3 E_{32}(-1)F_{01}(-1) - \tfrac 4 3 E_{31}(-2) + E_{31}(-1)H_{11}(-1) + \tfrac 1 3 a E_{10}(-1) + (k+1) aE_{10}(-1) \\ & & + 3 (2k+2) E_{32}(-1)F_{01}(-1) + 3 (k+1) E_{31}(-1)H_{01}(-1) -E_{32}(-1)F_{01}(-1) \\ & & -E_{31}(-1)H_{32}(-1) - k E_{31}(-2) \\ &=& 3 E_{32}(-1)F_{01}(-1) -2 E_{32}(-1)F_{01}(-1) -E_{32}(-1)F_{01}(-1) \\ & & - \tfrac 4 3 E_{31}(-2) +  \tfrac 4 3 E_{31}(-2) + \tfrac 1 3 a E_{10}(-1) - \tfrac 1 3 aE_{10}(-1) \\ & & + E_{31}(-1)H_{11}(-1) - E_{31}(-1)H_{01}(-1) - E_{31}(-1)H_{32}(-1) \\ &=& 0,
\end{eqnarray*}
where we drop $.\mathbf 1$ from the notation and use the equalities
\[ H_{11}= H_{10}+3H_{01} \quad \text{ and } \quad H_{32}= H_{10}+ 2H_{01} .\]

Assume that $k = -\frac 2 3$. We will continue to drop $.\mathbf 1$ from the notation. It again follows from Lemma \ref{lem-commute} that $E_{10}(0).v_k=0$ and $E_{01}(0).v_k=0$.  We now consider $F_{32}(1)$ and have \[ F_{32}(1).v_k = [F_{32}(1),u(v-w)]= [F_{32}(1),u](v-w)+u[F_{32}(1),v-w] .\]  We first compute $ [F_{32}(1),u](v-w)$. We use Lemma \ref{lem-commute-F}  and obtain:
\begin{eqnarray*}
& & [F_{32}(1),u](v-w) \\ &= &(k+ \tfrac{5}{3}) E_{10}(-1)(v-w) + E_{31}(-1)F_{21}(0)(v-w)\\
& &+\tfrac{2}{3} E_{21}(-1)F_{11}(0)(v-w) - E_{11}(-1)F_{01}(0)(v-w) - E_{10}(-1)H_{32}(0)(v-w)\\
&= &(k+ \tfrac{5}{3}) E_{10}(-1)(v-w) + E_{31}(-1)[F_{21}(0),v-w]\\
& &+\tfrac{2}{3} E_{21}(-1)[F_{11}(0),v-w] - E_{11}(-1)[F_{01}(0),v-w] - E_{10}(-1)[H_{32}(0),v-w].
\end{eqnarray*}

Now using Lemma \ref{lem-zero} and the fact that $H_{21}=2 H_{10}+ 3H_{01}$ along with the relation $[a,b] = 3w$, we obtain the following:
\begin{eqnarray*}
& & [F_{32}(1),u](v-w)\\ &=& \tfrac{2}{9}(k-\tfrac{1}{3}) a^3 E_{10}(-1) - k\cdot b\ a\ E_{10}(-1) -(3k+2) E_{10}(-1)\ c\\  
& &-(4k+\tfrac{8}{3}) w E_{10}(-1) -6 u E_{31}(-1) H_{01}(-1) - 6 u E_{32}(-1) F_{01}(-1)\\
& & -2 u E_{31}(-1) H_{10}(-1) +(2k+\tfrac 4 3) E_{31}(-1) a E_{21}(-2) + 3k \cdot b E_{31}(-2)  -(2k+\tfrac 2 3) a^2 E_{31}(-2)\\
&= & -\tfrac{2}{9} a^3 E_{10}(-1)+ \tfrac{2}{3} b a E_{10}(-1) - 6 u E_{31}(-1)H_{01}(-1)\\
& & - 6 u E_{32}(-1) F_{01}(-1) - 2 u E_{31}(-1) H_{10}(-1) + 2u E_{31}(-2)\\
&= & -\tfrac{2}{3} u a E_{10}(-1) - 6 u E_{31}(-1)H_{10}(-1)\\
& & - 6 u E_{32}(-1) F_{01}(-1) - 2 u E_{31}(-1) H_{10}(-1) + 2u E_{31}(-2),
\end{eqnarray*}
where the second equality is obtained by substituting $k=-\tfrac{2}{3}$.

Now we finally compute $u [F_{32}(1), v-w]$.  From Lemma \ref{lem-commute-F} and $H_{11}=H_{10}+3H_{01}$, we obtain:
\begin{eqnarray*}
u [F_{32}(1),v-w]&= & (6k+10) u\ E_{32}(-1)F_{01}(-1) + (k+\tfrac{4}{3}) u a E_{10}(-1)\\
& & + 2 u E_{31}(-1) H_{10}(-1) +(3k+8) u E_{31}(-1) H_{01}(-1) + (k - \tfrac{4}{3}) u E_{31}(-2)\\
&= & 6 u\ E_{32}(-1)F_{01}(-1) +\tfrac{2}{3} u a E_{10}(-1)\\
& & + 2 u E_{31}(-1) H_{10}(-1) + 6 u E_{31}(-1) H_{01}(-1) - 2 u E_{31}(-2),
\end{eqnarray*}
where we again substitute $k = - \tfrac{2}{3}$.  Now it is  clear that \[ F_{32}(1).v_k = [F_{32}(1),u]\ (v-w) + u\ [F_{32}(1),v-w] = 0 .  \]
\end{proof}

\section{Lemmas for Construction of Polynomials}

 The following results will be useful.  
\begin{Lem}\cite{P} \label{lem-product}
Let $X \in \frak g$ and let $Y_1,\dots, Y_m \in \mathcal{U}(\frak g)$.  Then \[(X^n)_L(Y_1 \dots Y_m)= \sum_{\substack{(k_1, \dots, k_m)\in(\mathbb{Z}_{+})^m\\ \sum{k_i=n}} }\binom{n}{k_1 \dots k_m}(X^{k_1})_L Y_1\dots (X^{k_m})_L Y_m,\]
where $\binom{n}{k_1\dots k_n} = \frac{n!}{k_1!\cdots k_m!}$.
\end{Lem}

\begin{proof}
This can be seen most readily by considering an exponential generating function.  
Given a derivation $D$ of $\mathcal{U}(\frak g)$, we may form the generating function \[ \exp(Dt)= 1 + Dt + \frac{D^2}{2!}t^2 + \cdots \in (\mathrm{End}\, \mathcal{U}(\frak g)) [[t]] .\]  Applying this to a $Y \in \mathcal{U}(\frak g)$, we obtain an element $\exp(Dt) Y \in \mathcal{U}(\frak g) [[t]]$.  The lemma is a direct consequence of the fact that $\exp(Dt)$ satisfies the identity \begin{equation}\label{eq-Lep}\exp(Dt)(Y_1 \cdots Y_n) = \exp(Dt)Y_1 \cdots \exp(Dt)Y_n .\end{equation} (See \cite{LL}.)  To obtain the lemma, replace $D$ with the adjoint action $X_L (= ad \, X)$ in the equation (\ref{eq-Lep}), and equate the coefficient of $t^n$ on both sides.  Finally, multiplying both coefficients by $n!$, we obtain the identity in the lemma.
\end{proof}

\begin{Lem}\label{lem-poly} \hfill
\begin{enumerate}
\item $(E_{ij}^m)_L (F_{ij}^m) \in m!H_{ij}(H_{ij}-1)\cdots(H_{ij}-m+1)+\mathcal{U}(\frak g)E_{ij}$, for all $i\alpha + j\beta \in \Delta_+$.
\item Suppose $X \in  \mathcal{U}(\frak g)_0$, the zero-weight subspace of $\mathcal U(\frak g)$.  Then $X \in \frak n_{-}\ \mathcal{U}(\frak g)$ if and only if $X \in \mathcal{U}(\frak g) \frak n_{+}$. 
\item  For $Y\in \mathcal{U}(\frak g)$ and $n>r>0$, we have \[(E_{ij}^n)_L (F_{ij}^r Y) \in F_{ij}\mathcal{U}(\frak g) + \tfrac{n!}{(n-r)!} (H_{ij}-n+r)\cdots(H_{ij}-n+1)(E_{ij}^{n-r})_L Y +  \mathcal{U}(\frak g)E_{ij}. \] 
\end{enumerate}
\end{Lem}

\begin{proof}
Part (1) follows from direct computation and part (2) follows by considering a PBW basis given in triangular form for $\mathcal{U}(\frak g)_0$.  For part (3), we consider an exponential generating function.  For simplicity, let us write $E, H, F$, in place of $E_{ij}, H_{ij}, F_{ij} $.  We then have: 
	
\begin{eqnarray}
\exp( (ad E) t) F^r Y  &=  (\exp( (ad E) t) F) ^r \ \exp( (ad E) t) Y \nonumber \\
&= ( F + H t - E t^2) ^ r \ \exp( (ad E) t) Y \label{eq-exp}
\end{eqnarray}
One can check
\begin{equation}  \label{eq-noexp}
( F + H t - E t^2)^ r   \in   F\mathcal{U}(\frak g)[[t]] + \sum_{i=0}^r \binom{r}{i}(-1)^i (H-i)(H-i-1)\cdots(H-r+1) E^i\ t^{r+i}
\end{equation}

 For convenience, we introduce the notation $(x)_{(i)} = x(x-1)\cdots (x-i+1)$ for $i>0$, and $(x)_{(0)}=1$.  Then we have 
\begin{align}
(x)_{(i)}&= (-1)^i (-x+i-1)_{(i)}, \label{eq-binom1}\\
(x+y)_{(m)}&= \displaystyle \sum_{i=0}^m \binom{m}{i}(x)_{(i)}(y)_{(m-i)}\label{eq-binom2}.
\end{align}
We  obtain the following identity using \eqref{eq-binom1} and\eqref{eq-binom2}:
\begin{eqnarray}
(x-n+r)_{(r)} &= &(-1)^r (n-r-(x-r+1))_{(r)} \nonumber \\
&=& (-1)^r  \sum_{i=0}^r \binom{r}{i}(n-r)_{(i)}(-(x-r+1))_{(r-i)} \nonumber\\
&= &(-1)^r \sum_{i=0}^r \binom{r}{i}(n-r)_{(i)}(-1)^{r-i}(x-i)_{(r-i)} \nonumber\\
&= &\sum_{i=0}^r \binom{r}{i}(-1)^i (n-r)_{(i)}(x-i)_{(r-i)} . \label{eqn-comb}
\end{eqnarray}

 Using this notation we combine equations \eqref{eq-exp}  and \eqref{eq-noexp} to write:
\[
\exp( (ad E) t) F^r Y  \in 
F\mathcal{U}(\frak g)[[t]] + \sum_{i=0}^r \binom{r}{i}(-1)^i (H-i)_{(r-i)} E^i \ t^{r+i}  \ \exp( (ad E) t) Y .
\]

Taking the coefficient of $t^n$ on both sides gives:
\begin{eqnarray} 
\frac{1}{n!}(ad E)^n (F^r Y) &\in& F\mathcal{U}(\frak g) +  \sum_{i=0}^r \binom{r}{i}(-1)^i (H-i)_{(r-i)}E^{i} \frac{1 }{(n-r-i)!} (ad E)^{n-r-i}  Y \nonumber \\ 
&\subseteq & F \mathcal{U}(\frak g) + \sum_{i=0}^r \binom{r}{i}(-1)^i  (H-i)_{(r-i)}
\frac{1}{(n-r-i)!}(ad E)^{n-r} Y +\mathcal{U}(\frak g)E   \label{eqn-ad}
\end{eqnarray}
With the substitution $x=H$, we obtain from (\ref{eqn-comb})
\begin{equation} \label{eqn-identity}  (H-n+r)_{(r)} = \sum_{i=0}^r \binom{r}{i} (-1)^{i} \frac{(n-r)!}{(n-r-i)!} (H-i)_{(r-i)} .\end{equation}
After multiplying (\ref{eqn-ad}) by $n!$, we use the identity (\ref{eqn-identity}) to obtain
\[ (ad E)^n (F^r Y) \in  F \mathcal{U}(\frak g) + \tfrac {n!}{(n-r)!} (H-n+r)_{(r)} (ad E)^{n-r} Y +\mathcal{U}(\frak g)E .\]
This proves part (3).
\end{proof}

The following lemmas will be needed for the construction of certain polynomials.

\begin{Lem} \label{lem-firststep}  The following identities hold in $\mathcal{U}(\frak g)$.
First we have:
\begin{eqnarray*}
 (F_{21}^2)_L [a] &= &-2 F_{21},\\
 (F_{21}^4)_L [b] &= & 4!(F_{31}F_{11}-F_{32}F_{10}),\\
 (F_{21}^6)_L [c] &= &-6!(F_{31}^2F_{01}-F_{32}F_{31}H_{01}-F_{32}^2 E_{01}),\ and\\
 (F_{21}^3)_L [a] &= &(F_{21}^5)_L [b]\ =\ (F_{21}^7)_L [c]= 0.
 \end{eqnarray*}
 Next we have: 
 \begin{eqnarray*}
  (F_{31}) _L [a] &= & F_{10},\\
 (F_{31}^2)_L [b] &= &-2!(F_{31}E_{11}-F_{21}E_{01}),\\
 (F_{31}^3)_L [c] &= & 3!(F_{31}(H_{32}+1)E_{01}+F_{32}E_{01}^2-F_{31}^2 E_{32}),\ and\\
 (F_{31}^2)_L [a] &= &(F_{31}^3)_L [b]\ =\ (F_{31}^4)_L [c]= 0.
\end{eqnarray*}
Finally we have:
\begin{eqnarray*}
  (F_{32}) _L [a] &= & F_{11},\\
 (F_{32}^2)_L [b] &= & 2!(F_{32}E_{10} - F_{21}F_{01}),\\
 (F_{32}^3)_L [c] &= &-3!(F_{32}F_{01}(H_{31}+2) + F_{31}F_{01}^2 - F_{32}^2 E_{31}),\ and\\
 (F_{32}^2)_L [a] &= &(F_{32}^3)_L [b]\ =\ (F_{32}^4)_L [c]= 0.
\end{eqnarray*}

\end{Lem}

\begin{proof}
Using Lemma \ref{lem-product}, we have:
\begin{align*} 
(F_{21}^4)_L [b] = &\phantom{LL} \tbinom{4}{3\ 1}(F_{21}^3)_L E_{31}({F_{21}}_L E_{11}) -\tbinom{4}{3\ 1}(F_{21}^3)_L E_{32} ({F_{21}}_L E_{10})\\ 
 & +\tbinom{4}{2\ 2}(F_{21}^2)_L E_{31}(F_{21}^2)_L E_{11} -\tbinom{4}{2\ 2}(F_{21}^2)_L E_{32} (F_{21}^2)_L E_{10}\\
= &\phantom{LL} \tbinom{4}{3\ 1}(6 F_{32} )(2 F_{10} ) -\tbinom{4}{3\ 1}(-6 F_{31} )(-2 F_{11} )\\
 & +\tbinom{4}{2\ 2}(-2 F_{11} )(-6 F_{31} ) -\tbinom{4}{2\ 2}(2 F_{10} ) (6 F_{32} )\\
= & \phantom{LL} 4! F_{31}F_{11} - 4! F_{32}F_{10}.
\end{align*}
We also have:
\begin{align*} 
(F_{21}^5)_L [b] = &\phantom{LL} \tbinom{5}{3\ 2}(F_{21}^3)_L E_{31}(F_{21}^2)_L E_{11} -\tbinom{5}{3\ 2}(F_{21}^3)_L E_{32} (F_{21}^2)_L E_{10}\\ 
= &\phantom{LL} \tbinom{5}{3\ 2}(6 F_{32} )(-6 F_{31} ) -\tbinom{5}{3\ 2}(-6 F_{31} )(6 F_{32} )\\
= &\phantom{LL} 0.
\end{align*}
The other cases are proved similarly.
\end{proof}

\begin{Lem} \label{lem-secondstep}
The following identities hold in $\mathcal{U}(\frak g)$.  First we have:
\begin{eqnarray*}
 \tfrac{1}{2}(E_{21}F_{21}^2)_L [a] & = & H_{21} ,\\
 \tfrac{1}{4!}(E_{21}^2F_{21}^4)_L [b] & \equiv & -2 H_{21}  \quad (\mathrm{mod} \ \mathcal{U}(\mathfrak{g})\mathfrak{n}_+) , \\ 
 \tfrac{1}{6!}(E_{21}^3F_{21}^6)_L [c] & \equiv & 3!H_{01}(H_{01} + 2)  \quad (\mathrm{mod} \ \mathcal{U}(\mathfrak{g})\mathfrak{n}_+) . 
\end{eqnarray*}
Next:
\[ (E_{10}F_{31}) _L [a] = H_{10} , \quad \tfrac{1}{2}(E_{10}^2F_{31}^2) _L [b] \equiv  \tfrac{1}{3!}(E_{10}^3F_{31}^3) _L [c] \equiv 0  \quad (\mathrm{mod} \ \mathcal{U}(\mathfrak{g})\mathfrak{n}_+) .
\]
Finally:
\begin{eqnarray*}
(E_{11}F_{32}) _L [a]& =&H_{11}, \\
\tfrac{1}{2}(E_{11}^2F_{32}^2) _L [b] & \equiv & -6 H_{01}  \quad (\mathrm{mod} \ \mathcal{U}(\mathfrak{g})\mathfrak{n}_+) , \\
\tfrac{1}{3!}(E_{11}^3F_{32}^3) _L [c] & \equiv & -6 H_{01}(H_{31}+2)  \quad (\mathrm{mod} \ \mathcal{U}(\mathfrak{g})\mathfrak{n}_+) . 
\end{eqnarray*}
\end{Lem}

\begin{proof}
Using Lemmas \ref{lem-product}, \ref{lem-firststep}, we have:
\begin{align*} 
\tfrac{1}{4!}(E_{21}^2 F_{21}^4)_L [b] = & \phantom{LL}(E_{21}^2)_L (F_{31}F_{11} - F_{32}F_{10})\\ 
 = & \phantom{LL} \tbinom{2}{2\ 0}((E_{21}^2)_L F_{31}) F_{11}  -\tbinom{2}{2\ 0}((F_{21}^2)_L F_{32}) F_{10} +\tbinom{2}{1\ 1}({F_{21}}_L F_{31})({F_{21}}_L F_{11}) \\ 
 &   -\tbinom{2}{1\ 1}({F_{21}}_L F_{32}) ({F_{21}}_L F_{10}) +\tbinom{2}{0\ 2}F_{31}(F_{21}^2)_L F_{11}  -\tbinom{2}{0\ 2} F_{32} (F_{21}^2)_L F_{10}\\ 
 = & \phantom{LL} \tbinom{2}{2\ 0}(-2 E_{11} ) F_{11}  -\tbinom{2}{2\ 0}(2 E_{10} ) F_{10} +\tbinom{2}{1\ 1}(-F_{10} )(-2 E_{10} ) \\ 
 &   -\tbinom{2}{1\ 1}(-F_{11}) (2 E_{11} ) +\tbinom{2}{0\ 2}F_{31}(-6 E_{31} )  -\tbinom{2}{0\ 2} F_{32} (6 E_{32} )\\ 
 = &  -2\tbinom{2}{2\ 0}(H_{11} + F_{11}E_{11} )  -2\tbinom{2}{2\ 0}(H_{10} + F_{10}E_{10} )+2 \tbinom{2}{1\ 1}F_{10}E_{10} \\ 
 &   +2\tbinom{2}{1\ 1}F_{11}E_{11} -6 \tbinom{2}{0\ 2}F_{31}6 E_{31}  -6\tbinom{2}{0\ 2} F_{32}E_{32}\\
 \equiv &  -2H_{11} -2 H_{10}  \equiv -2 H_{21}  \quad (\mathrm{mod} \ \mathcal{U}(\mathfrak{g})\mathfrak{n}_+) .
 \end{align*}
The other cases follow in the same way.
\end{proof}

\begin{Lem} \label{lem-thirdstep}
Suppose that $n,r,s,t \in \mathbb{Z}_{+}$ and $n=r+2s+3t$.  Then the following hold in $\mathcal{U}(\frak g)$:
\[ \begin{array}{l}
(E_{21}^n F_{21}^{2n})_L([a]^r[b]^s[c]^t) \\
\phantom{LLLL} \equiv  (-1)^r\tfrac{n!}{(n-r)!}\tfrac{(2n)!}{(2n-2r)!} (H_{21}-n+1) \cdots (H_{21}-n+r) (E_{21}^{n-r}F_{21}^{2(n-r)})_L ([b]^s[c]^t),  \\ \\
(E_{10}^n F_{31}^{n})_L([a]^r[b]^s[c]^t) \\
\phantom{LLLL}\equiv  \left(\tfrac{n!}{(n-r)!}\right)^2 (H_{10}-n+1) \cdots (H_{10}-n+r) (E_{10}^{n-r}F_{31}^{n-r})_L ([b]^s[c]^t)   , \\ \\
(E_{11}^n F_{32}^{n})_L([a]^r[b]^s[c]^t) \\
\phantom{LLLL} \equiv \left(\tfrac{n!}{(n-r)!}\right)^2 (H_{11}-n+1) \cdots (H_{11}-n+r) (E_{11}^{n-r}F_{32}^{n-r})_L ([b]^s[c]^t)    ,
\end{array} \]
where all the congruences are modulo $\mathcal{U}(\mathfrak{g})\mathfrak{n}_+$.
\end{Lem}

\begin{proof}
We prove only the first case.  Using Lemma \ref{lem-product} we have:
\begin{align*}
(F_{21}^{2n})_L([a]^r[b]^s[c]^t) = & \frac{(2n)!}{(2r)!(4s+6t)!} (F_{21}^{2r})_L([a]^r) (F_{21}^{4s+6t})_L([b]^s [c]^t)\\
= & \frac{(2n)!}{(2r)!(2n-2r)!} \frac{(2r)!}{(2!)^r}((F_{21}^{2})_L[a])^r (F_{21}^{2(n-r)})_L([b]^s [c]^t)\\
= & \frac{(2n)!}{2^r(2n-2r)!}((F_{21}^{2})_L[a])^r (F_{21}^{2(n-r)})_L([b]^s [c]^t)\\
= & \frac{(2n)!}{2^r(2n-2r)!}(-2F_{21})^r (F_{21}^{2(n-r)})_L([b]^s [c]^t)\\
= & (-1)^r\frac{(2n)!}{(2n-2r)!}(F_{21})^r (F_{21}^{2(n-r)})_L([b]^s [c]^t),
\end{align*}
since $(F_{21}^3)_L[a] = (F_{21}^5)_L [b] = (F_{21}^7)_L [c] = 0$ and  $(F_{21}^2)_L [a] = -2 F_{21}$ by Lemma \ref{lem-firststep}. 
 
Then we  use Lemma \ref{lem-poly} (3) with $Y= (F_{21}^{2(n-r)})_L([b]^s [c]^t)$ to obtain:
\begin{align*}
& 
(E_{21}^nF_{21}^{2n})_L([a]^r[b]^s[c]^t) \\ = & (-1)^r\tfrac{(2n)!}{(2n-2r)!}(E_{21}^n)_L\left( F_{21}^r (F_{21}^{2(n-r)})_L([b]^s [c]^t) \right)\\
\in & (-1)^r \tfrac{n!}{(n-r)!}\tfrac{(2n)!}{(2n-2r)!} (H_{21}-n+r)\cdots(H_{21}-n+1)(E_{21}^{n-r}F_{21}^{2(n-r)})_L ([b]^s [c]^t) + F_{21}\mathcal{U}(\frak g) + \mathcal{U}(\frak g)E_{21} .
\end{align*}

Now it follows from Lemma \ref{lem-poly} (2) that we have
\begin{align*}
& (E_{21}^nF_{21}^{2n})_L([a]^r[b]^s[c]^t) \\
 & \phantom{L} \equiv (-1)^r \tfrac{n!}{(n-r)!}\tfrac{(2n)!}{(2n-2r)!} (H_{21}-n+r)\cdots(H_{21}-n+1)(E_{21}^{n-r}F_{21}^{2(n-r)})_L ([b]^s [c]^t)   \quad (\mathrm{mod} \ \mathcal{U}(\mathfrak{g})\mathfrak{n}_+) .
\end{align*}

\end{proof}

 We give one more lemma.
 
 \begin{Lem} \label{lem-fourthstep}
 The following hold:
 
 \begin{eqnarray*}
 (E_{21}^4F_{21}^{8})_L ([b]^2)  &\equiv & 4!8! (2 H_{21}H_{11} + 2 H_{10}(H_{10}-1) - 6 H_{01}(H_{01}+1)) , \\
 (E_{21}^5F_{21}^{10})_L ([b][c]) &\equiv & -5!10! 2 H_{01}(H_{01}+2)(H_{21}-3) , \\
 (E_{10}^4F_{31}^4)_L ([b]^2) &\equiv &  (E_{10}^5F_{31}^5)_L ([b][c]) \equiv 0 ,\\
 (E_{11}^4F_{32}^4)_L ([b]^2) & \equiv  & (4!)^2\  2(9H_{01}(H_{01}-1)-H_{01}(H_{31}-3)), \\
 (E_{11}^5F_{32}^5)_L ([b][c]) & \equiv & (5!)^2\ 6 H_{01} (H_{01}-1) (H_{31}+2) ,
 \end{eqnarray*}
 where all the congruences are modulo $\mathcal{U}(\mathfrak{g})\mathfrak{n}_+$.
  \end{Lem}

\begin{proof}
We prove the first part only.  From Lemma \ref{lem-firststep}, we have:
\begin{align*}
(F_{21}^8)_L ([b]^2) = & \tbinom{8}{4\ 4} \left((F_{21}^4)_L[b]\right)^2  =  8!(F_{31}F_{11}-F_{32}F_{10})^2\\
  = & 8!(F_{31}^2F_{11}^2 - 2F_{32}F_{31}F_{11}F_{10} - 2F_{32}F_{31}F_{21} +F_{32}^2F_{10}^2) . 
\end{align*}
	
We thus obtain:
\begin{align*}
\tfrac{1}{8!}(E_{21}^4F_{21}^8)_L ([b]^2) = & (E_{21}^4)_L(F_{31}^2F_{11}^2 - 2F_{32}F_{31}F_{11}F_{10} - 2F_{32}F_{31}F_{21} +F_{32}^2F_{10}^2).
     \end{align*}
     This equals the following element modulo $\mathcal{U}(\mathfrak{g})\mathfrak{n}_+$:
\[ \begin{array}{ll}
 \phantom{L} \tbinom{4}{3 1 0 0}(-6E_{32})(-F_{10})F_{11}^2 & -2 \tbinom{4}{3 1 0 0}(6E_{31}) (-F_{10})F_{11}F_{10}\\
   + \tbinom{4}{3 0 1 0}(-6E_{32})F_{31}(-2E_{10})F_{11} & -2 \tbinom{4}{3 0 1 0}(6E_{31})F_{31}(-2E_{10})F_{10}\\
    + \tbinom{4}{2 2 0 0}(-2E_{11})(-2E_{11})F_{11}^2 & -2 \tbinom{4}{2 2 0 0} (2E_{10})(-2E_{11})F_{11}F_{10}\\
    + \tbinom{4}{2 1 1 0}(-2E_{11})(-F_{10})(-2E_{10})F_{11} & -2 \tbinom{4}{2 1 1 0} (2E_{10})(-F_{10})(-2E_{10})F_{10}\\
    + \tbinom{4}{2 0 2 0}(-2E_{11})F_{31}(-6E_{31})F_{11} & -2 \tbinom{4}{2 0 2 0} (2E_{10})F_{31}(-6E_{31})F_{10}\\
    -2\tbinom{4}{3 1 0 0}(6E_{31})(-F_{10})F_{21} &+\tbinom{4}{3 1 0 0}(6E_{31})(-F_{11})F_{10}^2\\
    - 2\tbinom{4}{3 0 1 0}(6E_{31})F_{31}H_{21} & + \tbinom{4}{3 0 1 0}(6E_{31})F_{32}(-2E_{11})F_{10}\\
    - 2\tbinom{4}{2 2 0 0}(2E_{10})(-2E_{11})F_{21} & + \tbinom{4}{2 2 0 0}(2E_{10})(2E_{10})F_{10}^2\\
    -2\tbinom{4}{2 1 1 0}(2E_{10})(-F_{10})H_{21} & + \tbinom{4}{2 1 1 0}(2E_{10})(-F_{11})(-2E_{11})F_{10}\\
  & + \tbinom{4}{2 0 2 0}(2E_{10})F_{32}(-6E_{32})F_{10},     \end{array} \]
where we have omitted the term $\tbinom{4}{1 3 0 0}(-F_{10})(-6E_{32})F_{11}^2 \in F_{10}\mathcal{U}(\mathfrak{g})$, which belongs to $\mathcal{U}(\mathfrak{g})\mathfrak{n}_+$ by Lemma \ref{lem-poly} (2), as well as similar terms which also belong to $\mathcal{U}(\mathfrak{g})\mathfrak{n}_+$.

We now see that $\tfrac{1}{4! 8!}(E_{21}^4F_{21}^8)_L ([b]^2)$ is equal to the following element modulo $\mathcal{U}(\mathfrak{g})\mathfrak{n}_+$:
\[\begin{array}{llll}
\phantom{LL} E_{32}F_{10}F_{10}^2 &+E_{31}F_{10}F_{11}F_{10} &+2E_{31}F_{10}F_{21}	&- E_{31}F_{11}F_{10}^2\\
+2E_{32}F_{31}E_{10}F_{11} &+4E_{31}F_{31}E_{10}F_{10} &-2E_{31}F_{31}H_{21}	&-2E_{31}F_{32}E_{11}F_{10}\\
+ E_{11}^2F_{11}^2 				&+2E_{10}E_{11}F_{11}F_{10} &+2E_{10}E_{11}F_{21}	&+ E_{10}^2F_{10}^2\\
-2E_{11}F_{10}E_{10}F_{11} &-4E_{10}F_{10}E_{10}F_{10} &+2E_{10}F_{10}H_{21}	&+2E_{10}F_{11}E_{11}F_{10}\\
+3E_{11}F_{31}E_{31}F_{11}	&+6E_{10}F_{31}E_{31}F_{10} 	&-3E_{10}F_{32}E_{32}F_{10}.
\end{array} \]
Again modulo $\mathcal{U}(\mathfrak{g})\mathfrak{n}_+$, this is equal to
\[ \begin{array}{l}6H_{01}							-4H_{10}						+2H_{21}				+2 H_{10}(H_{10}-1)
 +2H_{11}(H_{11}-1)	+4H_{31}H_{10}		\\	-2H_{21}H_{31}	
 -18H_{01}						+2H_{10}(H_{11}+1)	-4H_{10}				
 										-4H_{10}^2					+2H_{10}H_{21}.
\end{array} \]
After simplifying terms in this expression, we finally obtain:
\[
\tfrac{1}{4! 8!}(E_{21}^4F_{21}^8)_L ([b]^2) \equiv   2H_{21}H_{11} -6 H_{01}(H_{01}+1) + 2H_{10}(H_{10}-1)  \quad (\mathrm{mod} \ \mathcal{U}(\mathfrak{g})\mathfrak{n}_+) .
 \]
\end{proof}

\bibliographystyle{siam}

\end{document}